\documentclass[11pt]{amsart}

\usepackage{graphicx}
\usepackage{mathrsfs}
\usepackage{xcolor}
\usepackage{amssymb}
\usepackage{amsmath}
\usepackage{mathtools}
\usepackage{tikz}
\usetikzlibrary{arrows.meta,calc}
\usepackage{float}
\usepackage{enumitem}
\usepackage[margin=1.4in]{geometry}
\usepackage[hyperindex,breaklinks,colorlinks=true,
  linkcolor=blue!45!black,citecolor=green!35!black,
  urlcolor=blue!55!black]{hyperref}
\usepackage{aliascnt}
\usepackage[nameinlink,capitalise,noabbrev]{cleveref}

\newtheorem{prop}{Proposition}[section]
\newaliascnt{theo}{prop}
\newtheorem{theo}[theo]{Theorem}
\aliascntresetthe{theo}
\newaliascnt{lemm}{prop}
\newtheorem{lemm}[lemm]{Lemma}
\aliascntresetthe{lemm}
\newaliascnt{coro}{prop}
\newtheorem{coro}[coro]{Corollary}
\aliascntresetthe{coro}
\newaliascnt{exam}{prop}

\aliascntresetthe{exam}
\newaliascnt{rema}{prop}
\newtheorem{rema}[rema]{Remark}
\aliascntresetthe{rema}

\theoremstyle{definition}
\newaliascnt{defi}{prop}
\newtheorem{defi}[defi]{Definition}
\aliascntresetthe{defi}

\newtheorem*{theoremA}{Theorem A}
\newtheorem*{theoremB}{Theorem B}
\newtheorem*{theoremC}{Theorem C}
\newtheorem*{theoremD}{Theorem D}
\crefname{prop}{proposition}{propositions}
\crefname{theo}{theorem}{theorems}
\crefname{lemm}{lemma}{lemmas}
\crefname{coro}{corollary}{corollaries}
\crefname{exam}{example}{examples}
\crefname{rema}{remark}{remarks}
\crefname{defi}{definition}{definitions}

\DeclareMathOperator{\Ric}{Ric}
\DeclareMathOperator{\Hess}{Hess}
\DeclareMathOperator{\tr}{tr}
\DeclareMathOperator{\dist}{dist}
\DeclareMathOperator{\Lip}{Lip}
\newcommand{\bH}{\mathbb H}
\newcommand{\bB}{\mathbb B}
\newcommand{\bS}{\mathbb S}
\newcommand{\R}{\mathbb R}
\newcommand{\eps}{\varepsilon}
\newcommand{\md}{\mathrm d}

\title[Scalar--Mean Rigidity for Hyperbolic Balls]
{Scalar--Mean Rigidity for Hyperbolic Balls\\
in Dimensions $2\le n\le 5$  }
\author{Chenghang Lu}
\address{Albert-Ludwigs-Universit\"at Freiburg,
Mathematisches Institut, Ernst-Zermelo-Str.~1,
D-79104 Freiburg, Germany}
\email{chenghang.lu@math.uni-freiburg.de}
\thanks{C.L. is supported by the DFG through
IRTG 3132 (Grant No.~545046569).}

\date{\today}

\begin{document}
\raggedbottom

\begin{abstract}
In this paper, we prove the non-spin hyperbolic scalar--mean rigidity theorem
in dimensions $2\leq n\leq5$, resolving the corresponding low-dimensional
case of a problem of Gromov.  The proof combines the Jang deformation with
capillary $\mu$-bubble dimension reduction.  We also establish weighted
scalar--mean rigidity for Euclidean ball--torus products in total dimension
at most seven and show that its hyperbolic analogue fails.
\end{abstract}

\maketitle
\tableofcontents

\section{Introduction}\label{sec:introduction}

Scalar-curvature rigidity under contracting maps begins with Llarull's
theorem \cite[Theorems~B and C]{Llarull}.  It states,  that a
closed connected spin manifold admitting a nonzero-degree map to the unit
round sphere which is non-increasing on two-vectors cannot have scalar
curvature everywhere larger than that of the sphere; equality is rigid.
Thus a metric on $\bS^n$ which dominates the round metric on two-vectors and
satisfies $R_g\geq n(n-1)$ must be round.  Llarull's proof combines the
index of a twisted Dirac operator with the Bochner formula.  Listing extended
this method to compact symmetric spaces \cite{Listing}. For further generalizations of Llarull's theorem, see
\cite{Baer,CecchiniHankeSchick,LeeTam,LiSuWang,Listing}.\footnote{
For the distance non-increasing version of Llarull's theorem,
Wang--Wang--Xie--Zhu \cite[Theorem~1.1]{WWXZ} recently removed
the spin assumption in all dimensions.}

For manifolds with boundary, scalar curvature in the interior is coupled to
the mean curvature of the boundary.  The basic Euclidean model is the unit
ball, for which
\[
 R_{g_{\mathrm{Euc}}}=0,
 \qquad H_{\bS^{n-1}}=n-1.
\]
The corresponding boundary problem asks  a compact fill-in with
$R_g\geq0$, $H_{\partial M}\geq n-1$, and a nonzero-degree
$1$-Lipschitz boundary map to $\bS^{n-1}$ must be the Euclidean ball.

We use the differential norms introduced by Wang--Wang--Zhu
\cite[Definition~3.1]{WWZ}.  If $A:V\to W$ is a linear map between Euclidean
spaces with singular values
$\sigma_1(A)\geq\cdots\geq\sigma_r(A)\geq0$, set
\begin{equation}\label{eq:differential-norms}
 |A|:=\sigma_1(A),
 \qquad
 |A|_{\tr}:=\sum_{j=1}^r\sigma_j(A).
\end{equation}
Thus, for a map between $(m-1)$-manifolds,
\[
 |\mathrm dF|\leq |\mathrm dF|_{\tr}
 \leq(m-1)|\mathrm dF|.
\]

The following Euclidean scalar--mean rigidity theorems are due to
Wang--Wang--Zhu \cite{WWZ}; for the regularity theory of minimizing
capillary hypersurfaces, see
\cite{TaylorCapillary,DePhilippisMaggi,ChodoshEdelenLiCapillary,
DeMasiEdelenGasparettoLi}.  Wang--Zhang
\cite[Corollary~1.17]{WangZhang} extended the required regularity to ambient
dimension five, so both theorems hold for $2\leq m\leq5$.

\phantomsection\label{thm:wwz-fixed}
\begin{theoremA}[{\cite[Theorem~1.2]{WWZ},
with \cite[Corollary~1.17]{WangZhang}}]
Let $(M^m,g)$, $2\leq m\leq5$, be smooth, compact, connected, and oriented,
with connected nonempty boundary.  Suppose
\[
 R_g\geq0,
 \qquad H_{\partial M}\geq m-1.
\]
If $F:(\partial M,g_{\partial M})\to\bS^{m-1}$ is a smooth nonzero-degree
map satisfying $|\mathrm dF|\leq1$, then $F$ extends to an isometry
\[
 \widetilde F:(M,g)\longrightarrow\overline{\bB}^m.
\]
\end{theoremA}

\phantomsection\label{thm:wwz-listing}
\begin{theoremB}[{\cite[Theorem~1.5]{WWZ},
with \cite[Corollary~1.17]{WangZhang}}]
Let $(M^m,g)$, $2\leq m\leq5$, be smooth, compact, connected, and oriented,
with nonempty boundary satisfying $H_{\partial M}>0$ and $R_g\geq0$.  If
$F:\partial M\to\bS^{m-1}$ is smooth, has nonzero degree, and satisfies
\[
 H_{\partial M}\geq(m-1)|\mathrm dF|,
\]
then there is $c>0$ such that $F$ extends to an isometry
\[
 \widetilde F:(M,cg)\longrightarrow\overline{\bB}^m.
\]
\end{theoremB}

Let $\overline{\bB}^{\bH^n}_\rho$ denote the closed radius-$\rho$ ball in
hyperbolic $n$-space of sectional curvature $-1$.  Its boundary is the round
sphere $\bS^{n-1}(\sinh\rho)$ and has outward mean curvature
$(n-1)\coth\rho$.

We now state the main results of this paper.
\begin{theo}\label{thm:main}
Let $(M^n,g)$, $2\leq n\leq5$, be a compact, connected, oriented smooth
Riemannian manifold with connected nonempty boundary.  Fix $\rho>0$ and assume
\begin{equation*}
 R_g\geq-n(n-1),
 \qquad
 H_{\partial M}\geq(n-1)\coth\rho.
\end{equation*}
If there is a smooth map
\begin{equation*}
 f:(\partial M,g_{\partial M})\longrightarrow\bS^{n-1}(\sinh\rho)
\end{equation*}
with $\deg f\neq0$ and $\Lip(f)\leq1$, then $f$ extends to an isometry
\[
 \widetilde f:(M,g)\longrightarrow\overline{\bB}^{\bH^n}_\rho.
\]
\end{theo}

The relation
\[
 \coth\rho=\sqrt{1+(\sinh\rho)^{-2}}
\]
shows how the Euclidean boundary threshold is modified by negative
curvature.  For a closed connected oriented Riemannian $(n-1)$-manifold
$(\Sigma,g)$, define its hyperspherical radius by
\begin{equation}\label{eq:hyperspherical-radius}
 \operatorname{Rad}_{\bS^{n-1}}(\Sigma,g)
 :=\sup\left\{r>0:\begin{array}{c}
 \text{there is a Lipschitz map }
 F:(\Sigma,g)\to\bS^{n-1}(r),\\[-2pt]
 \deg F\neq0\text{ and }\Lip(F)\leq1
 \end{array}\right\}.
\end{equation}

\phantomsection\label{thm:btw-radius}
\begin{theoremC}[{\cite[Corollary~2]{BTW}
and \cite[Theorem~A]{AmmannLockman}}]
Let $(M^n,g)$, $n\geq3$, be a compact, connected, oriented Riemannian spin
manifold with connected boundary.  If $R_g\geq-n(n-1)$, then
\begin{equation}\label{eq:btw-radius}
 \inf_{\partial M} H_{\partial M}
 \leq(n-1)\sqrt{1+
 \operatorname{Rad}_{\bS^{n-1}}(\partial M,g_{\partial M})^{-2}}.
\end{equation}
Equality holds if and only if $(M,g)$ is isometric to a closed geodesic ball
$\overline{\bB}^{\bH^n}_\rho$, where
\[
 \sinh\rho
 =\operatorname{Rad}_{\bS^{n-1}}(\partial M,g_{\partial M}).
\]
\end{theoremC}
Gromov formulated the analogous rigidity problem for hyperbolic balls in
\cite[p.~110]{GromovFourLectures}.  Brendle--Tsiamis--Wang proved the spin
case \cite[Corollary~2]{BTW}.  Our first main theorem proves the non-spin
case in dimensions $2\leq n\leq5$.

In the dimensions considered here, \cref{thm:main} yields the same estimate
without a spin assumption.

\begin{coro}\label{cor:radius}
Let $(M^n,g)$, $2\leq n\leq5$, satisfy the topological and regularity
hypotheses of \cref{thm:main}, and suppose that
$R_g\geq-n(n-1)$.  Then
\begin{equation}\label{eq:nonspin-radius}
 \inf_{\partial M}H_{\partial M}
 \leq(n-1)\sqrt{1+
 \operatorname{Rad}_{\bS^{n-1}}(\partial M,g_{\partial M})^{-2}}.
\end{equation}
\end{coro}

Our second main result is the corresponding Listing form.

\begin{theo}\label{thm:listing}
Let $(M^n,g)$, $2\leq n\leq5$, satisfy the topological and regularity
hypotheses of \cref{thm:main}.  Assume
\[
 R_g\geq-n(n-1),
 \qquad H_{\partial M}>n-1.
\]
Let $F:(\partial M,g_{\partial M})\to\bS^{n-1}$ be a smooth map of nonzero
degree.  If
\begin{equation}\label{eq:listing-assumption}
 H_{\partial M}\geq(n-1)\sqrt{1+|\mathrm dF|^2}
 \quad\text{on }\partial M,
\end{equation}
then there is a unique $a>0$ such that the map
$aF:\partial M\to\bS^{n-1}(a)$ extends to an isometry
\[
 \widetilde F:(M,g)\longrightarrow
 \overline{\bB}^{\bH^n}_{\operatorname{arsinh}a}.
\]
Here the unit sphere is viewed in $\mathbb R^n$, so that $aF$ takes values in
$\bS^{n-1}(a)$.  Equality holds in
\eqref{eq:listing-assumption} everywhere.
\end{theo}

\begin{rema}\label{rem:untrapped}
The condition $\deg F\ne0$ does not imply that $\mathrm dF$ is nonzero at
every point.  At a critical point of $F$, \eqref{eq:listing-assumption} only
gives $H_{\partial M}\geq n-1$.  Hence \eqref{eq:listing-assumption} does not
imply $H_{\partial M}>n-1$ on $\partial M$.  For $3\leq n\leq5$, this strict
inequality is needed because the Dirichlet Jang theorem requires
$H_{\partial M}>|\tr_{\partial M}K|=n-1$ on $\partial M$.
\end{rema}

The weighted spin product problem was recently treated by Chow
\cite[Theorem~A]{Chow}; see also
Hao--Shi--Sun \cite{HaoShiSun}.

\phantomsection\label{thm:chow}
\begin{theoremD}
Let $3\leq N\leq7$ and $1\leq m\leq N-2$.  Let $(M^N,g)$ be closed,
connected, oriented, and spin, and let $\psi\in C^\infty(M)$ satisfy
\[
 -2\Delta\psi-|\nabla\psi|^2+R_g
 \geq (N-m)(N-m-1).
\]
Suppose that $\Phi:M\to\bS^{N-m}\times\mathbb T^m$ has nonzero degree and
its spherical projection is area non-increasing when $N-m\geq3$, and
$1$-Lipschitz when $N-m=2$.  Then $(M,g)$ is isometrically covered by
$\bS^{N-m}\times\R^m$.
\end{theoremD}

For manifolds with boundary, we obtain the following additional result.

\begin{theo}\label{thm:ball-torus}
Let $2\leq k\leq5$, $m\geq0$, and $N=k+m\leq7$.  Let $(M^N,g)$ be a
compact, connected, oriented Riemannian manifold with nonempty boundary, and
let $\varphi\in C^\infty(M)$.  Suppose that there are smooth maps
\[
 F:\partial M\longrightarrow\bS^{k-1},
 \qquad \Psi:M\longrightarrow\mathbb T^m,
\]
such that $(F,\Psi|_{\partial M})$ has nonzero degree and
\[
 -2\Delta\varphi-|\nabla\varphi|^2+R_g\geq0,
 \qquad
 H_{\partial M}+\langle\nabla\varphi,\eta\rangle\geq k-1,
 \qquad
 |\mathrm dF|\leq1.
\]
Then $\varphi$ is constant.  Moreover, for some flat metric $g_0$ on
$\mathbb T^m$, there is a boundary-preserving local isometry
\[
 \mathcal I:(M,g)\longrightarrow
 (\overline{\bB}^k\times\mathbb T^m,g_{\mathrm{Euc}}+g_0).
\]
Its first component extends $F$.  In particular,
$\widetilde F:=\pi_{\overline{\bB}^k}\circ\mathcal I$ satisfies
$\widetilde F|_{\partial M}=F$.
\end{theo}

\begin{rema}
The formally analogous hyperbolic ball--torus statement is false whenever
$k\geq2$ and $m\geq1$; see \cref{prop:hyperbolic-ball-torus-counterexample}.
\end{rema}

\noindent\textbf{Idea of the proof.}
The two-dimensional cases follow from normal-Jacobian comparison,
Gauss--Bonnet, and the Bol--Fiala isoperimetric inequality.  We therefore
describe the argument for $3\leq n\leq5$.

The proof is motivated by the relation between the asymptotically flat
positive mass theorem
\cite{SchoenYauPMT,SchoenYauPMTII,WittenPMT,
BiHaoHeShiZhu,BrendleWangPMT,BrendleWangSpacetime}
and the asymptotically hyperbolic positive mass theorem
\cite{WangAH,ChruscielHerzlich,ACG,HJM}.
Chru\'sciel--Delay observed that the asymptotically hyperbolic Riemannian
positive mass theorem follows from the asymptotically flat spacetime
positive mass theorem \cite[Section~3]{ChruscielDelay}.

We use the same reduction principle for compact manifolds with boundary.
The regularized Jang equation was introduced by Schoen--Yau
\cite[Section~4]{SchoenYauPMTII}.  The existence of smooth regularized
solutions and the boundary estimates follow from
\cite[Lemma~2.2]{EichmairPlateau}; the interior gradient estimate and
geometric Harnack inequality follow from
\cite[Lemmas~2.1 and~2.3]{EichmairPlateau}, with the correction in
\cite[Appendix~A]{EichmairMetzger}; and the local almost-minimizing estimates
and graph--cylinder description follow from
\cite[Propositions~6--7]{Eichmair}.  See Section~3 for details.  Related uses
of regularized Jang graphs appear in \cite{Sakovich,Lundberg,Tsang}.

The hyperbolic theorem is not a formal consequence of
\hyperref[thm:wwz-fixed]{Theorem~A}.  A constant rescaling cannot change
$R_g\geq-n(n-1)$ into nonnegative scalar curvature while preserving the
boundary metric.  Moreover, the hyperbolic boundary relation involves
\[
 \sqrt{H_{\partial M}^2-(n-1)^2}
\]
rather than $H_{\partial M}$ itself.

We set $K=g$ and solve Jang's equation with zero boundary data.  Then
\[
 2\mu=R_g+n(n-1)\geq0,
 \qquad J=0,
 \qquad \tr_{\partial M}K=n-1.
\]
The Schoen--Yau scalar-curvature identity contains the divergence term
$-2\operatorname{div}_{\bar g}X$, which produces a Robin boundary term, and
the boundary identity gives
\[
 \bar H-X(\bar\nu)
 \geq\sqrt{H_{\partial M}^2-(n-1)^2}.
\]
Thus the argument must be carried out at the level of the conformal
quadratic form.  A further difficulty is that the Jang graph may have
cylindrical ends.

To exclude these ends, we first construct a positive solution of the Robin
problem and obtain uniform lower bounds on it by a weighted capacity
estimate.  On compact domains bounded by distant end slices, a small Robin
perturbation gives conformal metrics of positive scalar curvature.  A
relative capillary $\mu$-bubble then carries the nonzero boundary degree into
one lower dimension.  The stability inequality and the projection to the
equatorial sphere contradict Gauss--Bonnet in dimension three and the
Wang--Wang--Zhu extremality theorem in dimensions four and five.  Hence the
Jang graph is compact.

On the compact graph, equality in the quadratic-form identity gives
$X=0$, $\mu=0$, $h=g$, and $R_{\bar g}=0$.
\hyperref[thm:wwz-fixed]{Theorem~A} identifies the Jang graph with a
Euclidean ball.  The equality $h=g$ yields a positive function $U$ such that
\[
 \Hess_gU=Ug,
 \qquad |\nabla U|^2=U^2-1.
\]
Its level sets recover the hyperbolic warped-product metric, and the boundary
radius determines the radius of the ball.  This proves the rigidity statement
and shows that the boundary map extends to an isometry.

For \cref{thm:ball-torus}, successive free-boundary weighted minimizers
remove the torus directions.  By \hyperref[thm:wwz-listing]{Theorem~B}, the
final slice is a Euclidean ball.  Weighted free-boundary foliations then
propagate equality through the preceding slices.

\noindent\textbf{Organization of this paper.}
Section~2 introduces capillary $\mu$-bubbles and the remaining preliminaries.
Section~3 proves the hyperbolic rigidity theorems for $3\leq n\leq5$.
Section~4 proves \cref{thm:ball-torus}.  Appendix~A gives the direct
two-dimensional proofs, and Appendix~B contains the hyperbolic warped-torus
counterexample.

\noindent\textbf{Acknowledgments.}
We thank Professors Guofang Wang and Chao Xia for introducing us to the study of
capillary hypersurfaces, for many helpful discussions, and for
his continued guidance and support.

\noindent\textbf{Disclosure on AI assistance.}
The author developed the ideas for the main theorems.
We used ChatGPT 5.6 sol auxiliary tools.
We had expected the hyperbolic analogue of
\cref{thm:ball-torus} to hold and had made several attempts
to prove it. The counterexample in
Appendix~\ref{app:warped-torus} was found by ChatGPT 5.6 sol.The author takes full
responsibility for the content of this paper.

\noindent\textbf{Related work.}
In September 2026, Wang--Wang--Xie--Zhu
\cite[Theorem~1.4]{WWXZ} developed a capillary version of
the singular dimension-descent method of Bi--Zhu \cite{BiZhu}
and used it to establish the all-dimensional extension of
the Euclidean scalar--mean rigidity theorem recalled as
\hyperref[thm:wwz-fixed]{Theorem~A}.
We then observed that combining their Euclidean theorem with
the Jang-equation blow-up argument developed here extends
the hyperbolic rigidity argument to dimensions $n\leq7$.
More recently, Jiang--Wang \cite{JiangWang} established
the conclusions of our \cref{thm:main,cor:radius}
in all dimensions $n\geq3$, using a similar Jang-equation reduction
together with a blow-up analysis of singular Jang graphs.
Our work was carried out independently of theirs.

\section{Preliminaries}\label{sec:inputs}

\noindent\textbf{Capillary $\mu$-bubbles.}
We use the variational formulas in \cite[Appendix~A]{WWZ} and the normal
convention of Wang--Xia \cite[Section~2]{WangXia}.  Let $(M^m,g)$ be oriented
with boundary $\partial M$, and let $\Omega\subset M$ be a finite-perimeter
set whose reduced boundary is smooth near $\partial M$.  Write
\[
 Y=\overline{\partial^*\Omega\cap\mathring M},
 \qquad Z=\partial Y\subset\partial M.
\]
For $h\in C^\infty(M)$ and $\beta\in C^\infty(\partial M)$ with
$|\beta|\leq1$, define the capillary $\mu$-bubble functional
\begin{equation}\label{eq:capillary-prelim-functional}
 \mathcal A_{h,\beta}(\Omega)
 =P(\Omega;\mathring M)-\int_\Omega h\,\mathrm dV
  -\int_{\partial^*\Omega\cap\partial M}\beta\,\mathrm dA.
\end{equation}
A minimizer of \eqref{eq:capillary-prelim-functional} is called a capillary
$\mu$-bubble.  At a smooth critical point, the first variation gives
\[
 H_Y=h,
 \qquad \cos J=\beta\quad\text{on }Z,
\]
where $J\in(0,\pi)$ is the contact angle.  If $\Omega$ is minimizing and
$|\beta|<1$ along $Z$, then
\begin{equation}\label{eq:capillary-stability}
\begin{split}
 Q_{h,\beta}(\phi,\psi)={}&
 \int_Y\left(\langle\nabla\phi,\nabla\psi\rangle
 -\bigl(\Ric_g(\nu,\nu)+|A_Y|^2+\partial_\nu h\bigr)\phi\psi\right)\\
 &+\int_Z\left(H_Z-\frac{H_{\partial M}}{\sin J}
 -\cot J\,H_Y
 -\frac{\partial_{\bar\nu}\beta}{\sin^2J}\right)\phi\psi.
\end{split}
\end{equation}
Stability means that $Q_{h,\beta}(\phi,\phi)\geq0$ for every
$\phi\in C^\infty(Y)$.

Here $\nu$ is the outward unit normal of $\Omega$ along $Y$, and $\bar N$
is the outward unit normal of $\partial M$ in $M$.  Along $Z$, $\mu$ is the
outward unit conormal in $Y$, while $\bar\nu$ is the unit normal in
$\partial M$ chosen so that $\{\nu,\mu\}$ and $\{\bar\nu,\bar N\}$ have the
same orientation.  The curvatures $H_Y$, $H_Z$, and $H_{\partial M}$ are
computed with respect to $\nu$, $\mu$, and $\bar N$, respectively.  These
normal fields satisfy
\begin{equation}\label{eq:capillary-angle}
 \cos J=-\langle\nu,\bar N\rangle=\langle\mu,\bar\nu\rangle,
 \qquad
 \mu=\sin J\,\bar N+\cos J\,\bar\nu,
 \qquad
 \nu=-\cos J\,\bar N+\sin J\,\bar\nu.
\end{equation}

\begin{figure}[H]
\centering
\begin{tikzpicture}[
 scale=0.92,
 >={Latex[length=2.8mm,width=2mm]},
 normal/.style={-{Latex[length=2.8mm,width=2mm]},very thick}
]
 \begin{scope}[xshift=-0.75cm]
 \coordinate (q) at (-3.35,0.05);
 \coordinate (p) at (2.05,0.18);
 \path[fill=orange!58]
   (-0.8,2.35)
   .. controls (-2.15,2.35) and (-3.25,1.35) .. (q)
   .. controls (-3.45,-1.25) and (-2.25,-2.25) .. (-0.65,-2.38)
   .. controls (0.85,-2.52) and (2.15,-1.55) .. (p)
   .. controls (2.0,1.25) and (0.85,2.3) .. (-0.8,2.35)
   -- cycle;
 \begin{scope}
   \clip
     (-0.8,2.35)
     .. controls (-2.15,2.35) and (-3.25,1.35) .. (q)
     .. controls (-3.45,-1.25) and (-2.25,-2.25) .. (-0.65,-2.38)
     .. controls (0.85,-2.52) and (2.15,-1.55) .. (p)
     .. controls (2.0,1.25) and (0.85,2.3) .. (-0.8,2.35)
     -- cycle;
   \path[fill=cyan!35]
     (q).. controls (-2.05,1.08) and (0.72,1.08) .. (p)
     -- (2.8,3)--(-4,3)--cycle;
   \fill[white,opacity=0.18] (-1.75,1.35) ellipse[x radius=1.0,y radius=0.55];
 \end{scope}
 \draw[very thick]
   (-0.8,2.35)
   .. controls (-2.15,2.35) and (-3.25,1.35) .. (q)
   .. controls (-3.45,-1.25) and (-2.25,-2.25) .. (-0.65,-2.38)
   .. controls (0.85,-2.52) and (2.15,-1.55) .. (p)
   .. controls (2.0,1.25) and (0.85,2.3) .. (-0.8,2.35);
 \draw[densely dashed,cyan!55!black]
   (q).. controls (-2.05,1.08) and (0.72,1.08) .. (p);
 \draw[densely dashed,cyan!55!black]
   (-0.8,2.3).. controls (-1.75,1.15) and (-1.65,-1.15) .. (-0.65,-2.33);

 \path[fill=green!52,opacity=0.9]
   (q).. controls (-2.05,1.08) and (0.72,1.08) .. (p)
      .. controls (0.75,-0.72) and (-2.15,-0.72) .. (q)--cycle;
 \draw[violet!85!black,densely dashed,very thick]
   (q).. controls (-2.05,1.08) and (0.72,1.08) .. (p);
 \draw[violet!85!black,very thick]
   (p).. controls (0.75,-0.72) and (-2.15,-0.72) .. (q);
 \node[green!25!black] at (-0.55,0.12) {$Y$};
 \node[violet!85!black] at (-0.85,-0.72)
   {$Z=\partial Y\subset\partial M$};
 \fill (p) circle (2.1pt) node[below left=2pt] {$p$};
 \node[cyan!35!black] at (-1.05,1.55) {$M\setminus\overline\Omega$};
 \node[orange!55!black] at (-0.95,-1.48) {$\Omega$};
 \node at (1.1,1.8) {$\partial M$};

 \draw[gray!65,thick]
   ($(p)+(90:0.78)$)--($(p)+(-90:0.78)$);
 \draw[normal,blue!75!black] (p)--($(p)+(0:0.9)$)
   node[pos=0.62,below=3pt] {$\bar N$};
 \draw[normal,teal!80!black] (p)--($(p)+(90:1.05)$)
   node[above] {$\bar\nu$};
 \draw[normal,violet!85!black] (p)--($(p)+(35:1.08)$)
   node[above right] {$\mu$};
 \draw[normal,red!80!black] (p)--($(p)+(125:1.08)$)
   node[above left] {$\nu$};
 \draw[thick]
   ($(p)+(35:0.42)$) arc[start angle=35,end angle=90,radius=0.42];
 \node at ($(p)+(62.5:0.64)$) {$J$};
 \end{scope}

 \draw[very thick,blue!75!black] (3.1,1.45)--(3.55,1.45);
 \node[anchor=west,text width=5.0cm,align=left,font=\small] at (3.65,1.45)
   {$\bar N$: outward unit normal of $\partial M$ in $M$.};
 \draw[very thick,teal!80!black] (3.1,0.58)--(3.55,0.58);
 \node[anchor=west,text width=5.0cm,align=left,font=\small] at (3.65,0.58)
   {$\bar\nu$: unit normal of $Z$ in $\partial M$.};
 \draw[very thick,violet!85!black] (3.1,-0.29)--(3.55,-0.29);
 \node[anchor=west,text width=5.0cm,align=left,font=\small] at (3.65,-0.29)
   {$\mu$: outward unit conormal of $Z$ in $Y$.};
 \draw[very thick,red!80!black] (3.1,-1.16)--(3.55,-1.16);
 \node[anchor=west,text width=5.0cm,align=left,font=\small] at (3.65,-1.16)
   {$\nu$: outward unit normal of $\Omega$ along $Y$.};
\end{tikzpicture}
\caption{A capillary hypersurface in a compact three-dimensional ball-type
domain.  The colors distinguish $\Omega$, $M\setminus\overline\Omega$, and
$Y$.  Here $Z=\partial Y\subset\partial M$, and the four vectors at
$p\in Z$ satisfy \eqref{eq:capillary-angle}.}
\label{fig:capillary-normals}
\end{figure}
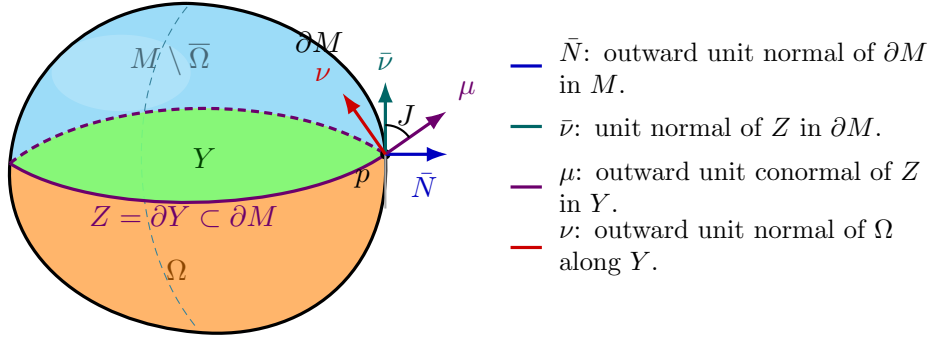

\section{The Jang equation and higher-dimensional rigidity}
\label{sec:jang-higher}

For a Riemannian manifold $(M,g)$, a symmetric $2$-tensor $K$, and
$u\in C^2(M)$, define the Jang operator by
\begin{equation*}
 \mathcal J(u)
 =\left(g^{ij}-\frac{\nabla^iu\nabla^ju}{1+|\nabla u|_g^2}\right)
  \left(\frac{\nabla^2_{ij}u}{\sqrt{1+|\nabla u|_g^2}}-K_{ij}\right).
\end{equation*}
The regularized Dirichlet problem is
\begin{equation}\label{eq:jang-regularized}
 \mathcal J(z_\tau)=\tau z_\tau\quad\text{in }M,
 \qquad z_\tau=0\quad\text{on }\partial M,
 \qquad 0<\tau\leq1.
\end{equation}

For $n\geq3$, define
\begin{equation*}
 c_n=\frac{4(n-1)}{n-2}.
\end{equation*}
For the initial data $(g,K)$, denote the energy and momentum densities by
\begin{equation*}
 2\mu=R_g+(\tr_gK)^2-|K|_g^2,
 \qquad
 J=\operatorname{div}_g\bigl(K-(\tr_gK)g\bigr).
\end{equation*}

\begin{prop}\label{prop:jang}
Let $(M^n,g)$, $3\leq n\leq5$, be compact and connected, and let
$K\in C^\infty(\operatorname{Sym}^2T^*M)$.  Suppose that $\partial M=\Sigma$
is connected and $H_\Sigma>|\tr_\Sigma K|$.  For some $\tau_i\downarrow0$,
the regularized solutions $z_{\tau_i}$ converge smoothly on compact subsets
of an open set $U$ to a solution $z$ of the Jang equation.  Here $U$ is the
component of the complement of the blow-up and blow-down regions that
contains a collar of $\Sigma$.  Denote the graph of $z$ by $G$; then
$z|_\Sigma=0$ and $\mathcal J(z)=0$.
Denote $\Gamma_{\mathrm{int}}=\partial U\cap\mathring M$.  If
$\Gamma_{\mathrm{int}}=\varnothing$, then $U=M$ and $G$ is compact.
Otherwise, $\Gamma_{\mathrm{int}}$ is a finite disjoint union of smooth
closed hypersurfaces, and the corresponding ends of $G$ are complete and
smoothly asymptotic to cylinders.

Write $\bar g=g+\mathrm dz^2$, $w=\sqrt{1+|\nabla z|_g^2}$,
$v=\nabla z/w$, and $h_{ij}=w^{-1}\nabla^2_{ij}z$.  Let the one-form $X$ on
$G$ be given by
\begin{equation}\label{eq:X-definition}
 X_i=(h_{ij}-K_{ij})v^j.
\end{equation}
Then
\begin{equation}\label{eq:jang-scalar}
 R_{\bar g}
 =2(\mu-J(v))+|h-K|_{\bar g}^2+2|X|_{\bar g}^2
  -2\operatorname{div}_{\bar g}X.
\end{equation}
Along $\Sigma$,
\begin{equation*}
 \bar H-X(\bar\nu)
 \geq\sqrt{H_\Sigma^2-(\tr_\Sigma K)^2}.
\end{equation*}
Here $\bar\nu$ is the outward $\bar g$-unit normal.
\end{prop}

\begin{figure}[t]
 \centering
 \includegraphics[width=\textwidth]{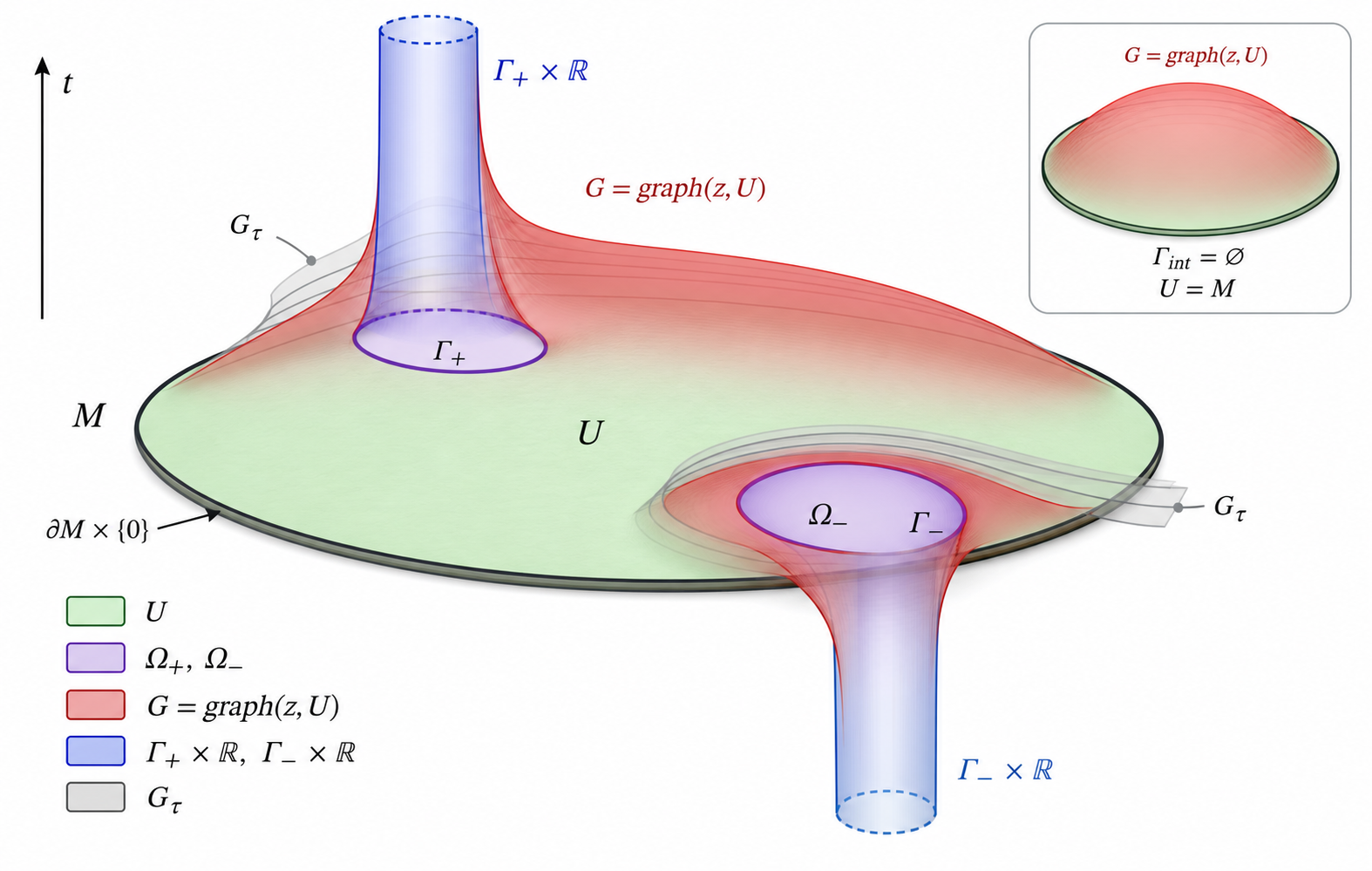}
 \caption{The graph--cylinder alternative in \cref{prop:jang}.  The limiting
 graph is defined over the green region $U$, and its ends over
 $\Gamma_\pm$ are asymptotic to $\Gamma_\pm\times\mathbb R$.}
 \label{fig:jang-graph}
\end{figure}

\begin{proof}
Extend $g$ and $K$ smoothly across $\Sigma$ and regard $M$ as a bounded
domain in a manifold without boundary.  Applied with $p=-K$ and zero
boundary value, \cite[Lemma~2.2]{EichmairPlateau} gives a smooth solution of
\eqref{eq:jang-regularized}.  The maximum principle gives
\[
 \tau\|z_\tau\|_{L^\infty(M)}
 \leq n\|K\|_{C^0(M)}.
\]
Since $H_\Sigma-|\tr_\Sigma K|$ has a positive minimum, the barriers in
\cite[Sections~3.2 and~3.4]{AEM} and boundary elliptic estimates give uniform
$C^k$ bounds, for every $k\geq0$, on a fixed collar of $\Sigma$.

Extend $K$ to $M\times\mathbb R$ by vertical translation and set its
$\partial_t$ components equal to zero.  For
$G_\tau=\operatorname{graph}(z_\tau)$,
\[
 |H_{G_\tau}|
 \leq|\tr_{G_\tau}K|+\tau|z_\tau|\leq C.
\]
By \cite[Propositions~6--7]{Eichmair}, together with
\cite[Lemmas~2.1--2.3]{EichmairPlateau} and the correction in
\cite[Lemma~A.1 and Remark~A.2]{EichmairMetzger}, a subsequence has a
geometric limit whose connected components are graphs or vertical
cylinders.  The collar estimates give smooth graphical convergence near
$\Sigma\times\{0\}$.

Write
\[
 \Omega_\pm=\{x\in\mathring M:z_{\tau_i}\to\pm\infty
      \text{ locally uniformly near }x\},
 \Omega_0=M\setminus
  \bigl(\overline{\Omega_+}\cup\overline{\Omega_-}\bigr).
\]
Let $U$ be the component of $\Omega_0$ containing a collar of $\Sigma$.  Then
\[
 z_{\tau_i}\longrightarrow z
 \quad\text{smoothly on compact subsets of }U,
 \qquad \mathcal J(z)=0,
 \qquad z|_\Sigma=0.
\]
Moreover, \cite[Proposition~7(c)]{Eichmair} shows that
$\Gamma_{\mathrm{int}}=\partial U\cap\mathring M$ is a finite disjoint union
of smooth closed hypersurfaces, $z$ tends to $+\infty$ or $-\infty$ along
each corresponding end, and vertical translates of the graph converge
smoothly to $\Gamma\times\mathbb R$ for each component $\Gamma$.  The ends
are complete, since the length of a curve is at least the variation of its
height.  If
$\Gamma_{\mathrm{int}}=\varnothing$, connectedness gives $U=M$.

On $G$, choose the unit normal $(\nabla z-\partial_t)/w$.  With the
conventions above, the Schoen--Yau identity
\cite[Equation~(10)]{Eichmair} is precisely
\eqref{eq:jang-scalar}.

For the boundary estimate, let $\nu$ be the outward $g$-unit normal of
$\Sigma$ and write
\[
 p=\partial_\nu z,
 \qquad P=\tr_\Sigma K.
\]
Since $z|_\Sigma=0$,
\[
 w=(1+p^2)^{1/2},
 \qquad
 v=\frac p w\nu,
 \qquad
 \bar\nu=\frac{\nu+p\partial_t}{w},
 \qquad \bar H=\frac{H_\Sigma}{w},
 \qquad
 \tr_\Sigma h=\frac{p}{w}H_\Sigma.
\]
The Jang equation along $\Sigma$ gives
\begin{equation}\label{eq:boundary-jang-split}
 0=\tr_\Sigma(h-K)+\frac1{w^2}(h-K)(\nu,\nu).
\end{equation}
Since the base projection of $\bar\nu$ is $\nu/w$,
\eqref{eq:X-definition} and \eqref{eq:boundary-jang-split} give
\begin{align*}
 X(\bar\nu)
 &=\frac{p}{w^2}(h-K)(\nu,\nu)\\
 &=-p\,\tr_\Sigma(h-K)
 =-\frac{p^2}{w}H_\Sigma+pP.
\end{align*}
Hence
\begin{equation}\label{eq:boundary-exact}
 \bar H-X(\bar\nu)=wH_\Sigma-pP.
\end{equation}
Finally,
\[
 (wH_\Sigma-pP)^2-(H_\Sigma^2-P^2)
 =(pH_\Sigma-wP)^2\geq0.
\]
Since $H_\Sigma>|P|$ and $w>|p|$, \eqref{eq:boundary-exact} is positive;
taking square roots completes the proof.
\end{proof}

\label{sec:jang-form}\label{sec:compactness}
Set $K=g$.  Then
\begin{equation}\label{eq:constraints}
 2\mu=R_g+(\tr_gK)^2-|K|_g^2
 =R_g+n(n-1)\geq0,
 \qquad J=0.
\end{equation}
On the Jang graph define
\begin{equation*}
 L=-c_n\bar\Delta+R_{\bar g},
 \qquad
 B=\frac{c_n}{2}\partial_{\bar\nu}+X(\bar\nu).
\end{equation*}
For compactly supported smooth $\phi$ and $\psi$, whose supports may meet
$\Sigma$, let
\begin{equation}\label{eq:q}
 Q(\phi,\psi)
 =\int_G\bigl(c_n\langle\bar\nabla\phi,\bar\nabla\psi\rangle
               +R_{\bar g}\phi\psi\bigr)
 +2\int_\Sigma X(\bar\nu)\phi\psi.
\end{equation}

\begin{lemm}\label{lem:q}
For every admissible $\phi$,
\begin{equation}\label{eq:q-exact}
\begin{split}
 Q(\phi,\phi)={}&\frac{2n}{n-2}\int_G|\bar\nabla\phi|^2
 +2\int_G|\bar\nabla\phi+\phi X|^2\\
 &+\int_G\bigl(2\mu+|h-g|^2\bigr)\phi^2\geq0.
\end{split}
\end{equation}
\end{lemm}

\begin{proof}
Insert \eqref{eq:jang-scalar} and \eqref{eq:constraints} into \eqref{eq:q}.
Integration by parts cancels the boundary contribution from
$-2\operatorname{div}X$ with
$2\int_\Sigma X(\bar\nu)\phi^2$.  Since
\[
 c_n|\nabla\phi|^2+2|X|^2\phi^2
 +4\phi\langle X,\nabla\phi\rangle
 =(c_n-2)|\nabla\phi|^2+2|\nabla\phi+\phi X|^2,
\]
and $c_n-2=2n/(n-2)$, the formula follows.
\end{proof}

For \cref{thm:main}, the boundary estimate is
\begin{equation}\label{eq:outer-fixed}
 \bar H-X(\bar\nu)
 \geq\sqrt{H_\Sigma^2-(n-1)^2}
 \geq\frac{n-1}{\sinh\rho}.
\end{equation}
Moreover, $\bar g_\Sigma=g_\Sigma$ because $z=0$ on $\Sigma$.

We next prove that the Jang graph is compact.

\begin{defi}\label{def:criticality}
Let $M$ be a connected manifold with compact boundary, and let $Q$ be the
symmetric bilinear form of a Robin pair $(\mathcal L,\mathcal B)$ on $M$.
Assume that $Q(\phi,\phi)\geq0$ for every $\phi\in C_c^\infty(M)$.  Here
$C_c^\infty(M)$ denotes the compactly supported functions that are smooth up
to $\partial M$; their supports may meet $\partial M$.
We say that a sequence $\{\xi_k\}\subset C_c^\infty(M)$ is a
\emph{null sequence} for $Q$ if $\xi_k\geq0$ and there exists a nonempty
open set $O\Subset\mathring M$ such that
\begin{equation}\label{eq:null-sequence}
 Q(\xi_k,\xi_k)\longrightarrow0,
 \qquad
 \int_O\xi_k^2=1\quad\text{for all }k.
\end{equation}
We say that $(\mathcal L,\mathcal B)$ is \emph{critical} if $Q$ admits a
null sequence, and \emph{subcritical} otherwise; see
\cite[Definition~6.3 and Theorem~6.4]{PinchoverVersano}.
\end{defi}

\begin{lemm}\label{lem:subcritical}
Assume $K=g$ and $R_g\geq-n(n-1)$.  If $G$ has a cylindrical Jang end,
then the pair $(L,B)$ is subcritical.
\end{lemm}

\begin{proof}
Suppose that the pair is critical and let $\xi_k$ be a null sequence.
By \eqref{eq:q-exact},
\[
 \|\bar\nabla\xi_k\|_{L^2(G)}\longrightarrow0,
 \qquad
 \|\bar\nabla\xi_k+\xi_kX\|_{L^2(G)}\longrightarrow0,
 \qquad
 \int_G(2\mu+|h-g|^2)\xi_k^2\longrightarrow0.
\]
Choose a connected compact exhaustion
$K_1\Subset K_2\Subset\cdots$ of $G$ with $\overline O\subset\mathring K_1$.
Poincar\'e's inequality on each $K_j$, followed by a diagonal argument, gives
a subsequence converging in $H^1_{\mathrm{loc}}(G)$ to a constant.  The
normalization in \eqref{eq:null-sequence} shows that this constant is positive;
write it as $c$.  Consequently
$X=0$ and $2\mu+|h-g|^2=0$ on $G$.  Thus $\mu=0$ and $h=g$.

Since $h=w^{-1}\Hess_gz$, one obtains
\begin{equation}\label{eq:hess-z}
 \Hess_gz=wg.
\end{equation}
Differentiating $w$ gives $\mathrm dw=\mathrm dz$, hence $w=z+C>0$; thus no
end can satisfy $z\to-\infty$.  On an end where $z\to+\infty$, set
$r=\operatorname{arcosh}w$.  Then
\[
 |\nabla r|_g
 =\frac{|\nabla w|_g}{\sqrt{w^2-1}}=1.
\]
Let $\Gamma$ be the compact hypersurface to which this end is asymptotic.  In a
$g$-normal collar of $\Gamma$, choose a unit-speed normal geodesic
$\gamma:(0,\varepsilon]\to U$ such that $\gamma(t)\to\Gamma$ as
$t\downarrow0$.  Since $z\to+\infty$ on this end,
\[
 z(\gamma(t))\longrightarrow\infty,
 \qquad
 r(\gamma(t))\longrightarrow\infty.
\]
For sufficiently small $t$, the function $r$ is smooth along $\gamma$.  Since
$|\nabla r|_g=1$,
\[
 \left|r(\gamma(t))-r(\gamma(\varepsilon))\right|
 \leq \int_t^\varepsilon |\nabla r|_g|\dot\gamma|_g\,\md s
 =\varepsilon-t\leq\varepsilon,
\]
which contradicts $r(\gamma(t))\to\infty$.
\end{proof}

\begin{defi}\label{def:uniformly-cylindrical}
Let $E$ be an end of $(M,g)$ and let $\mathcal L$ be a second-order elliptic
operator on $E$.  We say that $E$ is \emph{uniformly cylindrical for
$\mathcal L$} if a closed manifold $\Gamma$, a Riemannian metric
$g_\Gamma$ on $\Gamma$, and a diffeomorphism
\[
 \Phi:[0,\infty)\times\Gamma\longrightarrow E,
\]
exist such that the following estimates hold with a constant independent of
$s\geq0$.  Write $g_0=\md t^2+g_\Gamma$ and
$S_s=[s,s+1]\times\Gamma$.
\begin{enumerate}[label=\textnormal{(\roman*)},leftmargin=2.5em]
 \item $C^{-1}g_0\leq\Phi^*g\leq Cg_0$ on $[0,\infty)\times\Gamma$;
 \item the coefficients of $\Phi^*g$ and $(\Phi^*g)^{-1}$ have
 uniform $C^{2,\alpha}$ bounds on $S_s$;
 \item the coefficients of $\Phi^*\mathcal L$ have uniform
 $C^{0,\alpha}$ bounds on $S_s$.
\end{enumerate}
All coefficient bounds are taken in the product coordinates determined by
$\Phi$ and a fixed finite atlas on $\Gamma$.
\end{defi}

\begin{lemm}\label{lem:all-end}
Let $(M,g)$ be connected and complete with compact boundary.  Suppose that
a compact subset $M_0\subset M$ and finitely many ends
$E_1,\ldots,E_N$, with $N\geq1$, satisfy
\[
 M\setminus M_0=\bigsqcup_{j=1}^NE_j,
\]
and every $E_j$ is uniformly cylindrical for $\mathcal L$ in the sense of
\cref{def:uniformly-cylindrical}.  Let
\[
 \mathcal L=-c_n\Delta_g+V,
 \qquad
 \mathcal B=\frac{c_n}{2}\partial_\nu+\beta,
\]
where $V\in C^{0,\alpha}_{\mathrm{loc}}(M)$ and
$\beta\in C^{1,\alpha}(\partial M)$, and let the associated symmetric
bilinear form be
\[
  Q(\varphi,\psi)
  =\int_M(c_n\langle\nabla\varphi,\nabla\psi\rangle+V\varphi\psi)
  +2\int_{\partial M}\beta\varphi\psi.
\]
If $Q(\varphi,\varphi)\geq0$ for every
$\varphi\in C_c^\infty(M)$ and $(\mathcal L,\mathcal B)$ is subcritical, then
there is a positive $v\in C^{2,\alpha}_{\mathrm{loc}}(M)$ such that
\begin{equation}\label{eq:v}
 \mathcal Lv=0,
 \qquad \mathcal Bv=0,
\end{equation}
and $v$ has a positive lower bound on every end.
\end{lemm}

\begin{proof}
The proof follows Pinchover--Tintarev
\cite[Lemma~2.4]{PinchoverTintarev}, with the boundary terms treated using the
Robin condition.
By subcriticality, $Q$ admits no null sequence in the sense of
Definition~\ref{def:criticality}.  Fix the compact subset $M_0$ in the
statement, so that
\[
 M\setminus M_0=\bigsqcup_{j=1}^NE_j,
 \qquad E_j\simeq[0,\infty)\times\Gamma_j,
\]
with uniformly controlled product coordinates.  For $R>0$, let
\[
 M_R=M_0\cup\bigcup_{j=1}^N\bigl([0,R]\times\Gamma_j\bigr),
 \qquad
 Z_R=\bigsqcup_{j=1}^N\bigl(\{R\}\times\Gamma_j\bigr).
\]
Thus $M_R$ is compact and $\partial M_R=\partial M\sqcup Z_R$.

Consider the mixed eigenvalue problem
\[
 \begin{cases}
  \mathcal L\phi=\lambda\phi&\text{in }M_R,\\
  \mathcal B\phi=0&\text{on }\partial M,\\
  \phi=0&\text{on }Z_R.
 \end{cases}
\]
Its first eigenvalue is characterized by
\[
 \lambda_1(M_R)
 =\inf_{\substack{\phi\in H^1(M_R)\setminus\{0\}\\
                   \phi|_{Z_R}=0}}
   \frac{Q(\phi,\phi)}{\displaystyle\int_{M_R}\phi^2}.
\]
Since $Q(\phi,\phi)\geq0$, we have $\lambda_1(M_R)\geq0$.

We first show that $\lambda_1(M_R)>0$.  Suppose instead that
$\lambda_1(M_R)=0$, and let $\phi_R>0$ be a corresponding first
eigenfunction.  Choose a nonempty open set $O\Subset\mathring M_0$ and
normalize $\phi_R$ so that
\[
 \int_O\phi_R^2=1.
\]
Because $\phi_R=0$ on $Z_R$, its extension by zero from $M_R$ to $M$,
denoted by $\widetilde\phi_R$, belongs to $H^1(M)$ and has compact support.
The eigenvalue equation and the boundary conditions give
\[
 Q(\widetilde\phi_R,\widetilde\phi_R)=0.
\]
Choose nonnegative functions $\eta_k\in C_c^\infty(M)$ converging to
$\widetilde\phi_R$ in $H^1(M)$.  The local boundedness of the coefficients
and the continuity of the restriction map $H^1(M)\to L^2(\partial M)$ imply
\[
 Q(\eta_k,\eta_k)\longrightarrow0,
 \qquad
 \int_O\eta_k^2\longrightarrow1.
\]
Replacing $\eta_k$ by
\[
 \xi_k=
 \frac{\eta_k}{\left(\displaystyle\int_O\eta_k^2\right)^{1/2}},
\]
we obtain
\[
 \xi_k\geq0,
 \qquad
 \int_O\xi_k^2=1,
 \qquad
 Q(\xi_k,\xi_k)\longrightarrow0.
\]
Hence $\{\xi_k\}$ is a null sequence, contradicting subcriticality.
Therefore $\lambda_1(M_R)>0$.

Thus the mixed problem
\[
 \mathcal L\psi_R=0\quad\text{in }M_R,
 \qquad \mathcal B\psi_R=0\quad\text{on }\partial M,
 \qquad \psi_R=1\quad\text{on }Z_R
\]
has a positive solution.  Fix $o\in\mathring M_0$ and replace $\psi_R$ by
$\psi_R/\psi_R(o)$.  Harnack and Schauder estimates give a subsequence
converging to a positive function $\psi$ satisfying
\[
 \mathcal L\psi=0\quad\text{on }M,
 \qquad \mathcal B\psi=0\quad\text{on }\partial M,
 \qquad \psi(o)=1.
\]
For every compactly supported smooth function $\eta$, integration by parts
and the equations for $\psi$ give
\begin{equation}\label{eq:ground-transform}
 Q(\psi\eta,\psi\eta)=c_n\int_M\psi^2|\nabla\eta|^2.
\end{equation}
On each set $[t-1,t+1]\times\Gamma_j$, the Harnack and Schauder estimates
give
\begin{equation}\label{eq:psi-slab}
 \sup_{[t-1,t+1]\times\Gamma_j}
 (\psi+|\nabla\psi|+|\nabla^2\psi|)
 \leq C\inf_{[t-1,t+1]\times\Gamma_j}\psi.
\end{equation}
We now adapt the end construction of Li--Tam
\cite[Section~1 and the proof of Theorem~2.1]{LiTam} to the weighted operator
$\operatorname{div}(\psi^2\nabla\cdot)$; see also Grigor'yan
\cite[Sections~4.3 and~5]{Grigoryan}.
For $0\leq r<s$, define
\begin{equation*}
 A_j(t)=\int_{\Gamma_{j,t}}\psi^2|\nabla t|\,\mathrm dA,
 \qquad
 \mathcal R_j(r,s)=\int_r^s\frac{\mathrm dt}{A_j(t)},
\end{equation*}
and let
\[
 \operatorname{Cap}_{\psi,j}(r,s)
 :=\inf\left\{\int_{[r,s]\times\Gamma_j}\psi^2|\nabla\eta|^2:
 \eta|_{\Gamma_{j,r}}=1,\ \eta|_{\Gamma_{j,s}}=0\right\}.
\]
There is a constant $C\geq1$, independent of $j$, $r$, and $s$, such that
\begin{equation}\label{eq:weighted-capacity-comparison}
 C^{-1}\mathcal R_j(r,s)^{-1}
 \leq \operatorname{Cap}_{\psi,j}(r,s)
 \leq C\mathcal R_j(r,s)^{-1}.
\end{equation}
To prove the upper bound, take
\[
 \eta_0(t,\theta)
 =\frac{1}{\mathcal R_j(r,s)}\int_t^s\frac{\mathrm d\tau}{A_j(\tau)}.
\]
Then $\eta_0=1$ on $\Gamma_{j,r}$ and $\eta_0=0$ on
$\Gamma_{j,s}$.  The coarea formula gives
\[
 \int_{[r,s]\times\Gamma_j}\psi^2|\nabla\eta_0|^2
 =\int_r^s A_j(t)|\partial_t\eta_0|^2\,\mathrm dt
 =\mathcal R_j(r,s)^{-1},
\]
which proves the upper bound in \eqref{eq:weighted-capacity-comparison}.

For the lower bound, write
$\mathrm dV=J(t,\theta)\,\mathrm dt\,\mathrm d\theta$ in the product
coordinates.  The Harnack inequality on unit slabs and the metric bounds
imply
\[
 C^{-1}A_j(t)\leq\psi(t,\theta)^2J(t,\theta)\leq CA_j(t),
\]
with $C$ independent of $j$ and $t$.  If $\eta$ is admissible in the
definition of $\operatorname{Cap}_{\psi,j}(r,s)$, then Cauchy--Schwarz along
each curve $t\mapsto(t,\theta)$ gives
\[
 1=\left|\int_r^s\partial_t\eta\,\mathrm dt\right|
 \leq C\left(\int_r^s\psi^2|\nabla\eta|^2J\,\mathrm dt\right)^{1/2}
       \left(\int_r^s\frac{\mathrm dt}{A_j(t)}\right)^{1/2}.
\]
Squaring and integrating in $\theta$, we obtain
\[
 \int_{[r,s]\times\Gamma_j}\psi^2|\nabla\eta|^2
 \geq C^{-1}\mathcal R_j(r,s)^{-1}.
\]
Taking the infimum over $\eta$ proves the lower bound and hence
\eqref{eq:weighted-capacity-comparison}.

Let $m_j(t)=\min_{\Gamma_{j,t}}\psi$.  The Harnack inequality on unit slabs
gives
\begin{equation}\label{eq:A-m}
 C^{-1}m_j(t)^2\leq A_j(t)\leq Cm_j(t)^2.
\end{equation}
If $\inf_{E_j}\psi=0$, choose $t_k\to\infty$ with
$m_j(t_k)\to0$.  Passing to a subsequence, we may assume
$t_{k+1}-t_k\geq2$.  Then
$A_j(t)\leq Cm_j(t_k)^2$ for $|t-t_k|\leq1/2$.  Thus
\begin{equation}\label{eq:vanishing-implies-divergence}
 \inf_{E_j}\psi=0
 \quad\Longrightarrow\quad
 \mathcal R_j(0,\infty)=\infty.
\end{equation}
Suppose that $\mathcal R_j(0,\infty)=\infty$ for every end.  Since there are
only finitely many ends, we may choose $R_k\to\infty$ such that
$\mathcal R_j(0,R_k)\to\infty$ for every $j$.  Set $\eta_k=1$ on $M_0$ and,
on $E_j$, set
\[
 \eta_k(t,\theta)=
 \begin{cases}
  \displaystyle
  \frac{1}{\mathcal R_j(0,R_k)}
  \int_t^{R_k}\frac{\mathrm d\tau}{A_j(\tau)},&0\leq t\leq R_k,\\[6pt]
  0,&t\geq R_k.
 \end{cases}
\]
Each $\eta_k$ has compact support and belongs to $H^1(M)$.  Smooth compactly
supported functions approximate $\eta_k$ in $H^1(M)$ within a fixed compact
neighborhood of $\operatorname{supp}\eta_k$.  Since the coefficients of $Q$
are bounded there, \eqref{eq:ground-transform} remains valid for $\eta_k$ by
approximation.  The coarea formula gives
\[
 \int_M\psi^2|\nabla\eta_k|^2
 =\sum_j\mathcal R_j(0,R_k)^{-1}\longrightarrow0.
\]
Since $\eta_k=1$ on $M_0$, normalize $\psi\eta_k$ on a fixed relatively
compact open subset of $\mathring M_0$.  Equation
\eqref{eq:ground-transform} then gives a null sequence, contradicting
subcriticality.  Hence there is an end
$E_0$ such that $\mathcal R_0(0,\infty)<\infty$.  By
\eqref{eq:vanishing-implies-divergence},
\begin{equation}\label{eq:reference-end-lower}
 \inf_{E_0}\psi>0.
\end{equation}
By \eqref{eq:weighted-capacity-comparison},
$\operatorname{Cap}_{\psi,0}(0,R)$ is bounded below by a positive constant
independent of $R$.

Fix an end $E_j$ satisfying $\inf_{E_j}\psi=0$.  On $M_R$, minimize
\[
 \int_{M_R}\psi^2|\nabla h|^2
\]
over $h\in H^1(M_R)$ satisfying $h=1$ on $\Gamma_{j,R}$ and $h=0$ on
$\Gamma_{0,R}$ in the Sobolev sense.  Let $h_R$ be the minimizer.  Its
Euler--Lagrange equation is $\operatorname{div}(\psi^2\nabla h_R)=0$;
moreover,
$\partial_\nu h_R=0$ on $\partial M$ and on $\Gamma_{\ell,R}$ for
$\ell\notin\{0,j\}$.  Denote
\[
 \kappa_R=\int_{\Gamma_{j,R}}
 \psi^2\partial_\nu h_R\,\mathrm dA.
\]
The maximum principle and the Hopf lemma give $0<h_R<1$ and $\kappa_R>0$.
The divergence theorem and integration by parts give
\begin{equation}\label{eq:weighted-boundary-energy}
 \int_{\Gamma_{0,R}}\psi^2\partial_\nu h_R\,\mathrm dA=-\kappa_R,
 \qquad
 \int_{M_R}\psi^2|\nabla h_R|^2=\kappa_R.
\end{equation}
A competitor supported on $E_j$, together with
\eqref{eq:weighted-capacity-comparison}, yields
\[
 0<\kappa_R\leq C\mathcal R_j(0,R)^{-1}\longrightarrow0.
\]
Set $w_R=\kappa_R^{-1}h_R$.

Let $p_R$ minimize
\[
 \int_{[0,R]\times\Gamma_0}\psi^2|\nabla p|^2
\]
among functions equal to $1$ on $\Gamma_{0,0}$ and to $0$ on
$\Gamma_{0,R}$.  If $\nu$ is the outward unit normal to
$[0,R]\times\Gamma_0$ along $\Gamma_{0,R}$, then
\[
 \operatorname{Cap}_{\psi,0}(0,R)
 =-\int_{\Gamma_{0,R}}\psi^2\partial_\nu p_R\,\mathrm dA
 \geq\frac{c}{\mathcal R_0(0,\infty)}=:c_0>0.
\]
If $a_R=\min_{\Gamma_{0,0}}w_R$, the maximum principle gives
$w_R\geq a_Rp_R$ on $[0,R]\times\Gamma_0$.  Both sides vanish on
$\Gamma_{0,R}$, so the Hopf lemma and
\eqref{eq:weighted-boundary-energy} imply
\[
 1=-\int_{\Gamma_{0,R}}\psi^2\partial_\nu w_R\,\mathrm dA
 \geq
 -a_R\int_{\Gamma_{0,R}}\psi^2\partial_\nu p_R\,\mathrm dA
 =a_R\operatorname{Cap}_{\psi,0}(0,R).
\]
Thus $a_R\leq c_0^{-1}$.  A finite Harnack chain through $M_0$, followed by
the local Harnack inequality in the product coordinates on each end, gives
uniform bounds for $w_R$ on every compact subset of $M$.
Passing to a subsequence, $w_R\to w_j$ in $C^2_{\mathrm{loc}}$, where
\begin{equation}\label{eq:weighted-harmonic}
 \operatorname{div}(\psi^2\nabla w_j)=0,
 \qquad
 \partial_\nu w_j=0\quad\text{on }\partial M.
\end{equation}

On $E_j$, let $\nu_t$ point toward increasing $t$.  Applying the divergence
theorem to the region between $\Gamma_{j,t}$ and $\Gamma_{j,R}$, and then
letting $R\to\infty$, gives
\begin{equation}\label{eq:unit-normal-integral}
 \int_{\Gamma_{j,t}}\psi^2
 \partial_{\nu_t}w_j\,\mathrm dA=1
 \qquad (t>0).
\end{equation}
In particular, $w_j\not\equiv0$, and hence $w_j>0$.

The equation \eqref{eq:weighted-harmonic}, estimate
\eqref{eq:psi-slab}, and Harnack's inequality give
\[
 \sup_{\Gamma_{j,t}}|\nabla w_j|
 \leq C\inf_{[t-1,t+1]\times\Gamma_j}w_j.
\]
Consequently, \eqref{eq:unit-normal-integral} and \eqref{eq:A-m} imply
\[
 \inf_{[t-1,t+1]\times\Gamma_j}w_j
 \geq\frac{c}{A_j(t)}
 \geq\frac{c}{m_j(t)^2}.
\]
If $m_j(t)\leq1$, then
\[
 \inf_{\Gamma_{j,t}}\psi w_j
 \geq\frac{c}{m_j(t)}\geq c.
\]
Let $I=(a,b)$ be a component of $\{t:m_j(t)>1\}$ not containing $0$.
Since $\inf_{E_j}\psi=0$, the interval $I$ is bounded and
$m_j(a)=m_j(b)=1$.  At $t=a,b$, the Harnack inequality on unit slabs gives
$\sup_{\Gamma_{j,t}}\psi\leq C$, while
$\inf_{\Gamma_{j,t}}\psi w_j\geq c$.  Hence $w_j\geq c/C$ on
$\Gamma_{j,a}\cup\Gamma_{j,b}$.
The minimum principle for \eqref{eq:weighted-harmonic} gives the same bound
on $I\times\Gamma_j$.  Since $\psi>1$ there, $\psi w_j$ has a positive lower
bound on $I\times\Gamma_j$.  If a component of $\{t:m_j(t)>1\}$ has closure
meeting $0$, then its closure is compact because $\inf_{E_j}\psi=0$; the
continuous positive function $\psi w_j$ therefore has a positive minimum on
that closure.  Consequently
\begin{equation}\label{eq:corrected-end-lower}
 \inf_{E_j}\psi w_j>0.
\end{equation}
The identity
\[
 \mathcal L(\psi w_j)
 =-c_n\psi^{-1}\operatorname{div}(\psi^2\nabla w_j),
\]
and the condition $\partial_\nu w_j=0$ on $\partial M$ show that
$u_j:=\psi w_j$ satisfies $\mathcal Lu_j=0$ and $\mathcal Bu_j=0$.

Summing over the ends on which $\inf_{E_j}\psi=0$, set
\[
 v=\psi+\sum_{\{j:\,\inf_{E_j}\psi=0\}}u_j.
\]
It satisfies \eqref{eq:v}.  On every end with $\inf_{E_j}\psi>0$, its lower
bound follows from $v\geq\psi$; on every end with $\inf_{E_j}\psi=0$, it
follows from \eqref{eq:corrected-end-lower}.
\end{proof}

Applying \cref{lem:subcritical,lem:all-end} to the Jang graph, we obtain a
positive solution of $Lv=0$, $Bv=0$ satisfying $\inf_{E_j}v>0$ on every
end $E_j$.
In the present application, the metric and the coefficients of $L$ and $B$
are smooth.  Elliptic regularity therefore implies that $v$ is smooth.
\Cref{prop:jang} provides the required uniform estimates in the cylindrical
coordinates.

Assume that the Jang graph has ends.  For each end $E_j$, write
$\Phi_j:[0,\infty)\times\Gamma_j\to E_j$ for its product coordinates.  For
every sufficiently large integer $R$, set
\[
 G_R=G\setminus\bigcup_{j=1}^N
 \Phi_j((R,\infty)\times\Gamma_j),
 \qquad
 Z_R=\bigsqcup_{j=1}^N\Phi_j(\{R\}\times\Gamma_j).
\]
Then $G_R$ is compact and
\[
 \partial G_R=\Sigma\sqcup Z_R.
\]
Since $\bar g$ is uniformly equivalent to the product metric on every end,
there are constants $C,C_0>0$, independent of $R$, such that
\begin{equation}\label{eq:distance-to-ZR}
 \dist_{\bar g}(\Sigma,Z_R)
 \geq C^{-1/2}R-C_0\longrightarrow\infty
 \qquad\text{as }R\longrightarrow\infty.
\end{equation}
Let $\nu$ be the outward $\bar g$-unit normal to $Z_R$ in $G_R$, and set
\begin{equation*}
 \delta_R=\frac{c_n}{2}\frac{\partial_\nu v}{v}+X(\nu)
 \quad\text{on }Z_R.
\end{equation*}

\begin{lemm}
\label{lem:cylinderization}
For $\eps>0$, let
\begin{equation*}
\begin{split}
 Q_{R,\eps}(\phi,\psi)={}&
 \int_{G_R}\bigl(c_n\langle\bar\nabla\phi,\bar\nabla\psi\rangle
                  +R_{\bar g}\phi\psi\bigr)
 +2\int_\Sigma X(\bar\nu)\phi\psi\\
 &+2\int_{Z_R}(X(\nu)-\delta_R+\eps)\phi\psi.
\end{split}
\end{equation*}
Then
\begin{equation}\label{eq:picone}
 Q_{R,\eps}(\phi,\phi)
 =c_n\int_{G_R}v^2
  \left|\bar\nabla\left(\frac\phi v\right)\right|^2
 +2\eps\int_{Z_R}\phi^2.
\end{equation}
The principal eigenvalue $\lambda_{R,\eps}$ is positive.  Its positive
eigenfunction can be normalized so that, for fixed $R$,
$u_{R,\eps}/v\to1$ smoothly on $G_R$ as $\eps\downarrow0$.
\end{lemm}

\begin{proof}
Picone's identity for $Lv=0$ is
\[
 \int_{G_R}(c_n|\nabla\phi|^2+R_{\bar g}\phi^2)
 =c_n\int_{G_R}v^2|\nabla(\phi/v)|^2
 +c_n\int_{\partial G_R}\frac{\partial_\nu v}{v}\phi^2.
\]
At $\Sigma$, $\frac{c_n}{2}\partial_{\bar\nu}v+X(\bar\nu)v=0$, so the
boundary term cancels $2X(\bar\nu)\phi^2$.  At $Z_R$,
\[
 2(X(\nu)-\delta_R+\eps)
 =-c_n\frac{\partial_\nu v}{v}+2\eps.
\]
This proves \eqref{eq:picone}.  If $Q_{R,\eps}(\phi,\phi)=0$, then
$\phi=cv$ and $\phi|_{Z_R}=0$, hence $\phi=0$.  The first eigenvalue is
positive.  For $\eps=0$, the kernel is $\operatorname{span}\{v\}$.  Standard Robin
perturbation theory and elliptic regularity give smooth convergence.
\end{proof}

Choose $\eps_R$ so small that, with
$u_R=u_{R,\eps_R}$ and $\lambda_R=\lambda_{R,\eps_R}$,
\begin{equation}\label{eq:relative-eigenfunction-approximation}
 \sum_{j=0}^R\sup_{G_R}
 \frac{|\bar\nabla^j(u_R-v)|}{v}\leq R^{-1},
\end{equation}
and set
\begin{equation*}
 \widehat g_R=u_R^{4/(n-2)}\bar g.
\end{equation*}
The conformal formulas and the Robin condition on $\Sigma$ give
\begin{equation}\label{eq:conformal-formulas}
 R_{\widehat g_R}
 =u_R^{-(n+2)/(n-2)}Lu_R
 =\lambda_Ru_R^{-4/(n-2)}>0,
\end{equation}
and
\begin{equation}\label{eq:conformal-H}
 H^{\widehat g_R}_\Sigma
 =u_R^{-2/(n-2)}(\bar H-X(\bar\nu)).
\end{equation}
For \cref{thm:main}, \eqref{eq:outer-fixed} implies
\begin{equation}\label{eq:trace-fixed}
 H^{\widehat g_R}_\Sigma
 \geq\frac{n-1}{\sinh\rho}u_R^{-2/(n-2)}
 \geq\left|\mathrm d\left(\frac{f}{\sinh\rho}\right)
       \right|_{\tr,\widehat g_R}.
\end{equation}
The inequalities $\inf_{E_j}v>0$ for every $j$, together with positivity of
$v$ on $M_0$, give $\inf_Gv>0$.  Hence
\eqref{eq:relative-eigenfunction-approximation} gives
$u_R\geq c>0$, independently of $R$, and therefore
\[
 \dist_{\widehat g_R}(\Sigma,Z_R)\longrightarrow\infty.
\]
Uniform Harnack and Schauder estimates on the cylindrical slabs give
\[
 |\bar\nabla^jv|\leq C_jv\qquad (j\geq1).
\]
Together with \eqref{eq:relative-eigenfunction-approximation}, this gives
\[
 |\bar\nabla^j\log u_R|_{\bar g}\leq C_j
 \qquad (j\geq1),
\]
for each fixed $j$ and all sufficiently large $R$.
For a unit slab $S$, choose $x\in S$ and set
\[
 a_{R,S}:=u_R(x)^{2/(n-2)},
 \qquad
 g_{R,S}:=a_{R,S}^{-2}\widehat g_R
 =\left(\frac{u_R}{u_R(x)}\right)^{4/(n-2)}\bar g.
\]
Writing $g_0=\mathrm dt^2+g_\Gamma$, we have, in the product coordinates,
\[
 \begin{gathered}
 C^{-1}g_0\leq g_{R,S}\leq Cg_0,
 \qquad
 a_{R,S}\geq c^{2/(n-2)}=:a_0>0,\\
 \|g_{R,S}\|_{C^k(S)}
 +\|g_{R,S}^{-1}\|_{C^k(S)}\leq C_k.
 \end{gathered}
\]
Here the constants are independent of $R$ and $S$.
Standard harmonic-coordinate estimates for $g_{R,S}$ give interior
charts and boundary half-charts of a uniform radius $r_0>0$.
Since $\widehat g_R=a_{R,S}^2g_{R,S}$, the corresponding chart radius
for $\widehat g_R$ is at least $a_0r_0$, and
\[
 \begin{aligned}
 |\operatorname{Rm}_{\widehat g_R}|_{\widehat g_R}
 &=a_{R,S}^{-2}
   |\operatorname{Rm}_{g_{R,S}}|_{g_{R,S}}
 \leq Ca_0^{-2},\\
 |A_{Z_R}^{\widehat g_R}|_{\widehat g_R}
 &=a_{R,S}^{-1}
   |A_{Z_R}^{g_{R,S}}|_{g_{R,S}}
 \leq Ca_0^{-1}.
 \end{aligned}
\]
The second estimate applies on $S\cap Z_R$.
On the fixed compact part containing $\Sigma$, the same conclusions
follow from
\[
 \widehat g_R\longrightarrow v^{4/(n-2)}\bar g
 \qquad\text{in }C^\infty.
\]

\begin{prop}\label{prop:no-ends}
Let $3\leq n\leq5$ and assume the hypotheses of \cref{thm:main}.  Apply
\cref{prop:jang} with $K=g$, and denote its Jang graph by $G$.  Then $G$ is
compact.
\end{prop}

\begin{proof}
Suppose that $G$ has an end, choose $R_i\to\infty$, and set
\[
 M_i=G_{R_i},
 \qquad Z_i=Z_{R_i}.
\]
On the fixed collar of $\Sigma$, choose a smooth function $\chi$ such that
$\chi|_\Sigma=0$ and $\partial_{\bar\nu}\chi>0$ on $\Sigma$.  For each $i$,
choose $\tau_i>0$ sufficiently small and replace
$\widehat g_{R_i}$ by
\[
 g_i=(1+\tau_i\chi)^{4/(n-2)}\widehat g_{R_i}.
\]
The conformal formulas give $R_{g_i}\geq0$, preserve the boundary metric on
$\Sigma$, and make \eqref{eq:trace-fixed} strict.  We also choose $\tau_i$
so that the $C^i$-distance between $g_i$ and $\widehat g_{R_i}$ is at most
$i^{-1}$.  Thus the uniform geometry estimates obtained above remain valid,
the collar metrics form a precompact family in $C^\infty$, and
$\dist_{g_i}(\Sigma,Z_i)\to\infty$.

Set $F=f/\sinh\rho$ and
\[
 \delta_i=\min_\Sigma
 \left(H_\Sigma^{g_i}-|\mathrm dF|_{\tr,g_i}\right)>0.
\]
We follow the capillary dimension-reduction argument in the proof of
\cite[Proposition~3.2]{WWZ}.  By \cite[Lemma~3.1]{WWZ}, a sufficiently
small $C^1$ perturbation gives antipodal points
$p_-,p_+\in\bS^{n-1}$ and smooth maps
$\Phi_i:\Sigma\to\bS^{n-1}$ such that
\begin{equation}\label{eq:relative-flattened-gap}
 \deg\Phi_i=\deg F,
 \qquad
 H_\Sigma^{g_i}\geq|\mathrm d\Phi_i|_{\tr,g_i}+\varepsilon_i,
 \qquad \varepsilon_i>0.
\end{equation}
The maps $\Phi_i$ are constant near their inverse images of $p_-$ and
$p_+$, and the perturbations may be supported in fixed disjoint subsets of
$\Sigma$.

Let $s$ be the polar distance from $p_-$ and set
\[
 \beta_i=\langle\Phi_i,p_-\rangle=\cos(s\circ\Phi_i).
\]
Choose disjoint neighborhoods $\mathcal U_{i,-}$ and
$\mathcal U_{i,+}$ of $\Phi_i^{-1}(p_-)$ and $\Phi_i^{-1}(p_+)$ in the
fixed collar of $\Sigma$.  By \cite[Appendix~B, Claim~1]{WWZ} and
\cite[Steps~4--5 in the proof of Theorem~1.3]{WuCapillary}, the functional
\[
 \mathcal A_i(\Omega)
 =P(\Omega;\mathring M_i)
  -\int_{\partial^*\Omega\cap\Sigma}\beta_i\,\mathrm dA
\]
has a finite-perimeter minimizer $\Omega_i$ among sets satisfying
\[
 \chi_\Omega=1\quad\text{a.e. on }\mathcal U_{i,-},
 \qquad
 \chi_\Omega=0\quad\text{a.e. on }\mathcal U_{i,+}.
\]
The replacement argument in the same references gives neighborhoods
$\mathcal V_{i,\pm}$ such that
\[
 \mathcal U_{i,\pm}\Subset\mathcal V_{i,\pm},
 \qquad
 \chi_{\Omega_i}=1\ \text{a.e. on }\mathcal V_{i,-},
 \qquad
 \chi_{\Omega_i}=0\ \text{a.e. on }\mathcal V_{i,+}.
\]

All the prescribed neighborhoods are contained in the fixed collar.  Since
the pulled-back collar metrics form a precompact family in $C^\infty$, there
are smooth comparison domains $\Omega_i^0$ satisfying
\[
 \begin{aligned}
  \chi_{\Omega_i^0}&=1 &&\text{near }\mathcal U_{i,-},\\
  \chi_{\Omega_i^0}&=0 &&\text{near }\mathcal U_{i,+}.
 \end{aligned}
\]
and
\[
 P_{g_i}(\Omega_i^0;\mathring M_i)
 +\operatorname{Area}_{g_i}(\Sigma)\leq C.
\]
Since $|\beta_i|\leq1$, minimality gives
\begin{equation}\label{eq:relative-area-bound}
 P_{g_i}(\Omega_i;\mathring M_i)\leq C,
\end{equation}
where $C$ is independent of $i$.

Set $Y_i=\overline{\partial^*\Omega_i\cap\mathring M_i}$.  Since $Y_i$ is
disjoint from $\mathcal V_{i,-}\cup\mathcal V_{i,+}$, the angular projection
$P:\bS^{n-1}\setminus\{p_-,p_+\}\to\bS^{n-2}$ is defined on
$\Phi_i(\partial Y_i\cap\Sigma)$.  The relative intersection formula
\cite[Lemma~6.3]{Rade} (see also \cite[Lemma~3.2]{CWXZ}) gives
\[
 \sum_\alpha
 \deg\left(P\circ\Phi_i|_{\partial Y_{i,\alpha}\cap\Sigma}\right)
 =\pm\deg F\ne0,
\]
where $Y_{i,\alpha}$ are the connected components of $Y_i$.  We therefore
choose a component $Y_i^*$ satisfying
\begin{equation}\label{eq:first-slice-degree-relative}
 \deg\left(P\circ\Phi_i|_{\partial Y_i^*\cap\Sigma}\right)\ne0.
\end{equation}
The capillary regularity results
\cite[Theorem~1.1]{ChodoshEdelenLiCapillary} and
\cite[Corollary~1.17]{WangZhang} imply that $Y_i^*$ is smooth near
$\Sigma$.  Along its boundary in $\Sigma$ the contact angle $J$ satisfies
$\cos J=\beta_i$ and is bounded away from $0$ and $\pi$.

It remains to show that $Y_i^*$ is disjoint from $Z_i$.  Put
$r_i(x)=\dist_{g_i}(x,\Sigma)$.  The uniform geometry estimates obtained
above and the interior monotonicity formula give constants $r_0,c_0>0$,
independent of $i$, such that
\[
 \mathcal H^{n-1}(Y_i^*\cap B_{r_0}(x))
 \geq c_0r_0^{n-1}
\]
whenever $B_{2r_0}(x)\Subset\mathring M_i$.  If
$\overline{Y_i^*}\cap Z_i\ne\varnothing$, then connectedness and
\eqref{eq:first-slice-degree-relative} imply that $r_i(Y_i^*)$ contains
every value between $0$ and $\dist_{g_i}(\Sigma,Z_i)$.  We may therefore
choose points of $Y_i^*$ outside the fixed collar, with their $r_i$-values
separated by $4r_0$ and with distance at least $2r_0$ from $Z_i$.  Since
$r_i$ is $1$-Lipschitz, the corresponding $r_0$-balls are disjoint, and the
preceding estimate yields
\[
 \mathcal H^{n-1}(Y_i^*)
 \geq c\,\dist_{g_i}(\Sigma,Z_i).
\]
This contradicts \eqref{eq:relative-area-bound} for large $i$.  Thus
\begin{equation}\label{eq:first-slice-avoids-Z}
 \overline{Y_i^*}\cap Z_i=\varnothing.
\end{equation}

Consequently $Y_i^*$ is a smooth compact stable capillary minimal
hypersurface with boundary in $\Sigma$.  We now apply the first- and
second-variation computation in the proof of
\cite[Proposition~3.2]{WWZ}.  If $\mu$ is the outward unit conormal of
$\partial Y_i^*$ in $Y_i^*$ and $\bar\nu$ is the corresponding unit normal
in $\Sigma$, this computation gives the symmetric bilinear form
\begin{equation}\label{eq:relative-first-slice-form}
 Q_i(\zeta,\xi)=
 \int_{Y_i^*}\left(\langle\nabla\zeta,\nabla\xi\rangle
                    +\frac12R_{Y_i^*}\zeta\xi\right)
 +\int_{\partial Y_i^*}\mathcal B_i\zeta\xi,
\end{equation}
with $Q_i(\zeta,\zeta)\geq0$, where
\[
 \mathcal B_i
 =H_{\partial Y_i^*}-\frac{H_\Sigma^{g_i}}{\sin J}
  +\frac{\langle\nabla s,\mathrm d\Phi_i(\bar\nu)\rangle}{\sin J}.
\]
Moreover, \eqref{eq:relative-flattened-gap} gives
\begin{equation}\label{eq:relative-descended-gap}
 H_{\partial Y_i^*}-\mathcal B_i
 \geq
 \left|\mathrm d\left(P\circ\Phi_i|_{\partial Y_i^*}\right)\right|_{\tr}
 +\frac{\varepsilon_i}{\sin J}.
\end{equation}

If $n=3$, insert $1$ into \eqref{eq:relative-first-slice-form}.  By
Gauss--Bonnet and \eqref{eq:first-slice-degree-relative},
\[
 0\leq2\pi\chi(Y_i^*)
 -\int_{\partial Y_i^*}
  \left|\mathrm d\left(P\circ\Phi_i\right)\right|
 -\varepsilon_i\int_{\partial Y_i^*}\frac1{\sin J}<0,
\]
a contradiction.  If $4\leq n\leq5$, the first eigenfunction of the
associated Robin problem and the conformal change used in the proof of
\cite[Proposition~3.2]{WWZ} give a metric $\widetilde g$ on $Y_i^*$ such
that
\[
 R_{\widetilde g}\geq0,
 \qquad
 H_{\partial Y_i^*}^{\widetilde g}>
 \left|\mathrm d\left(P\circ\Phi_i|_{\partial Y_i^*}\right)
 \right|_{\tr,\widetilde g}.
\]
Since $\dim Y_i^*=n-1\leq4$, \cite[Proposition~3.2]{WWZ} contradicts
\eqref{eq:first-slice-degree-relative}.
\end{proof}

\begin{proof}[Proof of \cref{thm:main} for $3\leq n\leq5$]
By \cref{prop:no-ends}, the Jang graph is compact.  By
the graph--cylinder alternative in \cref{prop:jang}, this forces
$\Gamma_{\mathrm{int}}=\varnothing$; hence the graph is defined by a smooth
function $z$ on all of $M$.

Let $\lambda_1$ be the first eigenvalue of $(L,B)$ on the compact graph.
By \eqref{eq:q-exact}, $\lambda_1\geq0$.  Suppose that $\lambda_1>0$, let
$u>0$ be a first eigenfunction, and set
$\widehat g=u^{4/(n-2)}\bar g$.  Then
\[
 R_{\widehat g}=\lambda_1u^{-4/(n-2)}>0,
 \qquad
 H^{\widehat g}_{\partial M}
 \geq\left|\mathrm d\left(\frac{f}{\sinh\rho}\right)
       \right|_{\tr,\widehat g}.
\]
Choose $\chi\in C^\infty(M)$ supported near $\partial M$ such that
$\chi=0$ and $\partial_{\widehat\nu}\chi>0$ on $\partial M$, where
$\widehat\nu$ is the outward unit normal with respect to $\widehat g$.
For $v_\tau=1+\tau\chi$ and
$g_\tau=v_\tau^{4/(n-2)}\widehat g$, the conformal formulas give
\[
 R_{g_\tau}
 =v_\tau^{-(n+2)/(n-2)}
 \left(R_{\widehat g}+\tau\left(R_{\widehat g}\chi
 -\frac{4(n-1)}{n-2}\Delta_{\widehat g}\chi\right)\right)
\]
and
\[
 g_\tau|_{\partial M}=\widehat g|_{\partial M},
 \qquad
 H^{g_\tau}_{\partial M}
 =H^{\widehat g}_{\partial M}
  +\frac{2(n-1)}{n-2}\tau\partial_{\widehat\nu}\chi
 >H^{\widehat g}_{\partial M}.
\]
For sufficiently small $\tau>0$, $R_{g_\tau}\geq0$.  Wang--Wang--Zhu
\cite[Proposition~3.2]{WWZ}, together with the five-dimensional regularity
of Wang--Zhang \cite[Corollary~1.17]{WangZhang}, then implies that
$\deg f=0$, a contradiction.  Hence $\lambda_1=0$.

Let $u>0$ be a first eigenfunction.  It satisfies $Lu=0$, $Bu=0$, and
\[
 Q(u,u)=\lambda_1\int_Gu^2=0.
\]
Equality in \eqref{eq:q-exact} gives
\begin{equation}\label{eq:jang-equality}
 \nabla u=0,
 \qquad X=0,
 \qquad \mu=0,
 \qquad h=g.
\end{equation}
The scalar identity \eqref{eq:jang-scalar} then gives $R_{\bar g}=0$.

Scale $\bar g$ by $(\sinh\rho)^{-2}$.  Equations
\eqref{eq:outer-fixed} and \eqref{eq:jang-equality} give
\[
 H^{(\sinh\rho)^{-2}\bar g}_{\partial M}\geq n-1.
\]
The map
\[
 \frac{f}{\sinh\rho}:
 (\partial M,(\sinh\rho)^{-2}\bar g|_{\partial M})
 \longrightarrow\bS^{n-1}
\]
has nonzero degree and is distance non-increasing.  By
\hyperref[thm:wwz-fixed]{Theorem~A}, it extends to an isometry from
$(G,(\sinh\rho)^{-2}\bar g)$ onto the Euclidean unit ball.  After
rescaling, $f$ extends to an isometry
\[
 \widehat f:(G,\bar g)\longrightarrow
 \overline{\bB}^n(\sinh\rho).
\]
In particular,
\[
 (\partial M,g|_{\partial M})
 \cong(\bS^{n-1},\sinh^2\!\rho\,g_{\bS^{n-1}}).
\]

The equality $h=g$ gives \eqref{eq:hess-z}.  Since
$\mathrm dw=\mathrm dz$, the function $w-z$ is constant.  Set $U=w$.
Then
\begin{equation}\label{eq:obata-system}
 \Hess_gU=Ug,
 \qquad
 |\nabla U|_g^2=U^2-1.
\end{equation}
Moreover, $z=0$ on the connected boundary, so $U$ is constant on
$\partial M$; write this value as $U_0$.

The first equation in \eqref{eq:obata-system} shows that $U$ has no
interior maximum, while it cannot be constant.  Thus $U\leq U_0$ and its
minimum lies in the interior.  At every critical point,
\eqref{eq:obata-system} gives $U=1$ and $\Hess U=g$.  Hence every critical
point is an isolated nondegenerate minimum and $U_0>1$.  Also
$|\nabla U|=\sqrt{U_0^2-1}>0$ on $\partial M$.

For $1<c<U_0$, put $\Omega_c=\{U\leq c\}$.  If $c-1>0$ is sufficiently
small, $\Omega_c$ is a disjoint union of balls, one about each critical
point.  There are no critical values in $(1,U_0)$, so these sublevel sets
are mutually diffeomorphic.  When $c$ is sufficiently close to $U_0$,
$M\setminus\Omega_c$ is a collar of the connected boundary; hence
$\Omega_c$ is connected.
Consequently $U$ has a unique critical point $o$, and $U(o)=1$.

On $M\setminus\{o\}$ define
\[
 s=\operatorname{arcosh}U,
 \qquad s_0=\operatorname{arcosh}U_0.
\]
Since $\mathrm dU=\sinh s\,\mathrm ds$, equations
\eqref{eq:obata-system} imply
\[
 |\nabla s|=1,
 \qquad
 \Hess s=\coth s\,(g-\mathrm ds^2).
\]
The integral curves of $\nabla s$ are unit-speed geodesics.  The backward
integral curve through $x\ne o$ reaches $o$ and has length $s(x)$, whereas
$|\nabla s|=1$ gives $s(x)\leq d_g(o,x)$.  Hence
$s=d_g(o,\cdot)$.  In particular, $s$ has no cut point before
$\partial M$.

Let $g_s$ denote the induced metric on $\{s=\mathrm{constant}\}$.  The
second equation above gives
\[
 \partial_s g_s=2\coth s\,g_s.
\]
It follows that
\[
 g=\mathrm ds^2+\sinh^2s\,h
\]
for a fixed metric $h$ on
$\bS^{n-1}$, and smoothness at $o$ in geodesic polar coordinates gives
$h=g_{\bS^{n-1}}$.  Therefore
\[
 g=\mathrm ds^2+\sinh^2s\,g_{\bS^{n-1}}.
\]
Since the induced boundary metric is
$\sinh^2s_0\,g_{\bS^{n-1}}$ and was identified above with
$\sinh^2\rho\,g_{\bS^{n-1}}$, we have $s_0=\rho$.  Thus
\[
 (M,g)\cong
 \overline{\bB}^{\bH^n}_\rho.
\]
Finally, $\widehat f|_{\partial M}=f$ is an isometry between the boundary
spheres.  Such an isometry extends to the corresponding hyperbolic balls;
composing this extension with the isometry just obtained gives the required
extension of $f$.
\end{proof}

\begin{proof}[Proof of \cref{thm:listing} for $3\leq n\leq5$]
Since $H_{\partial M}>n-1$, \cref{prop:jang}
applies with $K=g$; denote the function defining the graph by $z$.  Then
$z|_{\partial M}=0$, so $\bar g_{\partial M}=g_{\partial M}$.  The scalar
identity \eqref{eq:jang-scalar}, the quadratic identity \eqref{eq:q-exact},
and the construction of the positive Robin solution remain valid.  The
boundary estimate becomes
\begin{equation}\label{eq:listing-jang}
 \bar H-X(\bar\nu)
 \geq\sqrt{H_{\partial M}^2-(n-1)^2}
 \geq(n-1)|\mathrm dF|.
\end{equation}
The left-hand side is positive because $H_{\partial M}>n-1$.

If the graph has ends, use the conformal metrics
$\widehat g_R=u_R^{4/(n-2)}\bar g$ constructed from
\eqref{eq:relative-eigenfunction-approximation}.  Equations
\eqref{eq:conformal-formulas}, \eqref{eq:conformal-H}, and
\eqref{eq:listing-jang} give
\begin{equation}\label{eq:listing-conformal}
 H^{\widehat g_R}_{\partial M}
 \geq(n-1)|\mathrm dF|_{\widehat g_R}
 \geq|\mathrm dF|_{\tr,\widehat g_R}.
\end{equation}
Choose $\chi$ supported in a collar of $\partial M$ such that
\[
 \chi|_{\partial M}=0,
 \qquad
 \partial_{\widehat\nu_R}\chi>0
 \quad\text{on }\partial M.
\]
For fixed $R$ and sufficiently small $\tau>0$, set
\[
 v_\tau=1+\tau\chi>0,
 \qquad
 g_\tau=v_\tau^{4/(n-2)}\widehat g_R.
\]
The conformal formulas give
\[
 \begin{aligned}
 R_{g_\tau}
 &=v_\tau^{-(n+2)/(n-2)}
   \left(R_{\widehat g_R}v_\tau
         -c_n\tau\Delta_{\widehat g_R}\chi\right)>0,\\
 g_\tau|_{\partial M}
 &=\widehat g_R|_{\partial M}.
 \end{aligned}
\]
Moreover, on $\partial M$,
\[
 \begin{aligned}
 H_{\partial M}^{g_\tau}
 &=H_{\partial M}^{\widehat g_R}
   +\frac{c_n}{2}\tau\partial_{\widehat\nu_R}\chi\\
 &>(n-1)|\mathrm dF|_{g_\tau}
 \geq|\mathrm dF|_{\tr,g_\tau}.
 \end{aligned}
\]
The capillary argument in the proof of \cref{prop:no-ends},
with $F$ replacing $f/\sinh\rho$, now gives $\deg F=0$,
a contradiction. Thus $G$ is compact, and \cref{prop:jang} yields
\[
 G=\operatorname{graph}(z),
 \qquad z\in C^\infty(M).
\]

Let $\lambda_1$ be the first Robin eigenvalue on the compact graph.  If
$\lambda_1>0$, its positive first eigenfunction defines a conformal metric
$\widetilde g$ satisfying
\[
 R_{\widetilde g}>0,
 \qquad
 H^{\widetilde g}_{\partial M}
 \geq(n-1)|\mathrm dF|_{\widetilde g}
 \geq|\mathrm dF|_{\tr,\widetilde g}.
\]
The same conformal perturbation near $\partial M$ makes the last inequality
strict without changing the boundary metric.  This contradicts
\hyperref[thm:wwz-listing]{Theorem~B}.  Hence $\lambda_1=0$, and equality in
\eqref{eq:q-exact} gives
\[
 X=0,
 \qquad \mu=0,
 \qquad h=g.
\]
By \eqref{eq:jang-scalar}, \eqref{eq:listing-jang}, and $X=0$,
\[
 \begin{gathered}
 R_{\bar g}=0,
 \qquad
 H_{\partial M}^{\bar g}
 \geq\sqrt{H_{\partial M}^2-(n-1)^2}>0,\\
 H_{\partial M}^{\bar g}
 \geq(n-1)|\mathrm dF|_{\bar g}.
 \end{gathered}
\]
Thus \hyperref[thm:wwz-listing]{Theorem~B} gives
\[
 \widehat F:(G,c\bar g)\xrightarrow{\cong}\overline{\bB}^n,
 \qquad
 c>0,
 \qquad
 \widehat F|_{\partial M}=F.
\]
Setting $a=c^{-1/2}$, we obtain
\[
 a\widehat F:(G,\bar g)
 \xrightarrow{\cong}\overline{\bB}^n(a),
 \qquad
 (aF)^*g_{\bS^{n-1}(a)}
 =\bar g_{\partial M}=g_{\partial M}.
\]
Moreover, $h=g$ and $z|_{\partial M}=0$ imply, for $U=w$,
\[
 \begin{gathered}
 U=z+U_0,
 \qquad U|_{\partial M}=U_0,\\
 \Hess_gU=Ug,
 \qquad
 |\nabla U|_g^2=U^2-1.
 \end{gathered}
\]
The argument following \eqref{eq:obata-system} therefore yields
\[
 (M,g)\cong
 \overline{\bB}^{\bH^n}_{\operatorname{arsinh}a},
 \qquad
 H_{\partial M}=(n-1)\sqrt{1+a^{-2}},
 \qquad
 |\mathrm dF|=a^{-1}.
\]
Hence equality holds in \eqref{eq:listing-assumption}, and
the boundary isometry $aF$ extends to
\[
 \widetilde F:(M,g)\xrightarrow{\cong}
 \overline{\bB}^{\bH^n}_{\operatorname{arsinh}a},
 \qquad
 \widetilde F|_{\partial M}=aF.
\]
 
\end{proof}

\section{The ball--torus theorem}\label{sec:ball-torus}

We use the following consequence of Brendle--Hung
\cite[Propositions~4.1--4.2]{BrendleHung}.

\begin{lemm}
\label{lem:one-step-weighted-slicing}
Let $(M^d,g)$, $3\leq d\leq7$, be compact, connected, and oriented, let
$\phi\in C^\infty(M)$, and let $\theta:M\to\bS^1$ be smooth.  Suppose that
$\alpha$ is a closed $(d-2)$-form on $\partial M$ such that
\[
 \int_{\partial M}\alpha\wedge
 \theta^*\mathrm{dvol}_{\bS^1}\ne0,
 \qquad
 H_{\partial M}+\langle\nabla\phi,\eta\rangle>0.
\]
Then there is a compact, connected, embedded, oriented hypersurface
$\Sigma\subset M$ that minimizes
$\int_\Sigma e^\phi\,\mathrm dA_g$ in its relative homology class, meets
$\partial M$ orthogonally, and satisfies
\[
 \int_{\partial\Sigma}\alpha\ne0.
\]
Moreover, a positive function $u\in C^\infty(\Sigma)$ and a constant
$\lambda\geq0$ can be chosen so that
\[
\begin{split}
 -\Delta_\Sigma u
 &-\Ric_g(\nu,\nu)u-|A_\Sigma|^2u
 +\Hess_g\phi(\nu,\nu)u\\
 &\qquad-\langle\nabla_\Sigma\phi,\nabla_\Sigma u\rangle
 =\lambda u,
 \qquad
 \partial_\eta u=h_{\partial M}(\nu,\nu)u.
\end{split}
\]
If $\psi=\phi|_\Sigma+\log u$, then
\[
\begin{split}
 -2\Delta_\Sigma\psi-|\nabla_\Sigma\psi|^2+R_\Sigma
 ={}&
 \bigl(-2\Delta_M\phi-|\nabla\phi|^2+R_g\bigr)|_\Sigma\\
 &+|\nabla_\Sigma\log u|^2+|A_\Sigma|^2+2\lambda,
\end{split}
\]
and
\[
 H_{\partial\Sigma}
 +\langle\nabla_\Sigma\psi,\eta\rangle
 =
 H_{\partial M}+\langle\nabla\phi,\eta\rangle.
\]
\end{lemm}

\begin{proof}
Fiber integration gives
\[
 \int_{\partial M}\alpha\wedge\theta^*\mathrm{dvol}_{\bS^1}
 =\int_{\bS^1}
   \left(\int_{\partial(\theta^{-1}(s))}\alpha\right)
   \mathrm{dvol}_{\bS^1}(s).
\]
Sard's theorem therefore allows us to choose $s\in\bS^1$ that is a regular
value of both $\theta$ and $\theta|_{\partial M}$ and satisfies
\[
 \int_{\partial(\theta^{-1}(s))}\alpha\ne0.
\]
Set $\widetilde\Sigma=\theta^{-1}(s)$.
Set $\widehat g=e^{2\phi/(d-1)}g$.  Then
\[
 H_{\partial M}^{\widehat g}
 =e^{-\phi/(d-1)}
  \bigl(H_{\partial M}+\langle\nabla\phi,\eta\rangle\bigr)>0.
\]
Apply the minimizing-current construction in
\cite[Appendix~B, Theorem~B.1]{BrendleHung} to
$(M,\widehat g)$, $\widetilde\Sigma$, and $\alpha$.  In dimension $d\leq7$,
its support is a smooth embedded free-boundary hypersurface.  Choose a
connected component $\Sigma$ such that
$\int_{\partial\Sigma}\alpha\ne0$.  This component is itself minimizing in
its relative homology class.  Indeed, in the minimizing current obtained
there, replacing one copy of this component by a
smaller homologous competitor would strictly decrease the total mass.  Since
$\mathrm dA_{\widehat g}=e^\phi\mathrm dA_g$ on hypersurfaces, $\Sigma$
minimizes the weighted area.

By \cite[Proposition~4.1(vi)]{BrendleHung}, the positive first eigenfunction
of the weighted stability operator has nonnegative eigenvalue and satisfies
the equation and boundary condition in the lemma.
By \cite[Propositions~4.1(ix) and~4.2]{BrendleHung}, the scalar-curvature and
boundary mean-curvature identities in the lemma hold.
\end{proof}

\begin{lemm}\label{lem:weighted-ball}
Let $(M^k,g)$, $2\leq k\leq5$, be compact, connected, and oriented, and let
$\phi\in C^\infty(M)$.  Suppose that $F:\partial M\to\bS^{k-1}$ has nonzero
degree and
\[
 -2\Delta\phi-|\nabla\phi|^2+R_g\geq0,
 \qquad
 H_{\partial M}+\langle\nabla\phi,\eta\rangle\geq k-1,
 \qquad
 |\mathrm dF|\leq1.
\]
Then $\phi$ is constant and $F$ extends to an isometry
\[
 \widetilde F:(M,g)\longrightarrow\overline{\bB}^k.
\]
\end{lemm}

\begin{proof}
Suppose first that $k=2$, and write
$\kappa_{\partial M}=H_{\partial M}$ for the geodesic curvature of the
boundary.  The degree and Lipschitz assumptions give
\[
 \operatorname{Length}_g(\partial M)\geq2\pi.
\]
By Gauss--Bonnet,
\begin{align*}
0
&\leq\int_M
 \bigl(-2\Delta\phi-|\nabla\phi|^2+R_g\bigr)\,\mathrm dA_g\\
 &=4\pi\chi(M)
  -2\int_{\partial M}
   \bigl(\kappa_{\partial M}+\langle\nabla\phi,\eta\rangle\bigr)\,\mathrm ds_g
  -\int_M|\nabla\phi|^2\,\mathrm dA_g\\
&\leq4\pi\chi(M)-2\operatorname{Length}_g(\partial M)\leq0.
\end{align*}
Since a connected oriented surface with nonempty boundary has $\chi(M)\leq1$,
equality holds throughout.  Hence $M$ is a disk, $\phi$ is constant,
$R_g=0$, $\kappa_{\partial M}=1$, and
$\operatorname{Length}_g(\partial M)=2\pi$.  Since $M$ is simply connected
and flat, it admits an orientation-preserving local isometry
$D:M\to\mathbb R^2$.  The parametrized plane curve $D|_{\partial M}$ has
constant signed curvature $1$ and length $2\pi$.  Hence its image is a unit
circle and $D|_{\partial M}$ is a diffeomorphism onto that circle.

For $y\notin D(\partial M)$, the degree $\deg(D,M,y)$ is the winding number
of $D|_{\partial M}$ about $y$.  Since $D$ is an orientation-preserving local
diffeomorphism, this degree is the number of points in $D^{-1}(y)$.  It is
$1$ inside the circle and $0$ outside.  Consequently, $D$ is an isometry
from $(M,g)$ onto the Euclidean unit disk.

Finally,
\[
 2\pi\leq2\pi|\deg F|
 \leq\int_{\partial M}|\mathrm dF|\,\mathrm ds_g
 \leq\operatorname{Length}_g(\partial M)=2\pi.
\]
Thus $|\deg F|=1$ and $|\mathrm dF|\equiv1$.  Therefore
$F\circ(D|_{\partial M})^{-1}$ is an isometry of the unit circle and extends
to a Euclidean isometry of the unit disk.  Composing this extension with
$D$ gives the required extension of $F$.

Suppose now that $3\leq k\leq5$.  Define
$\widetilde g=e^{2\phi/(k-1)}g$.  Then
\[
 R_{\widetilde g}
 =e^{-2\phi/(k-1)}
 \left(-2\Delta\phi-|\nabla\phi|^2+R_g
       +\frac1{k-1}|\nabla\phi|^2\right)\geq0,
\]
and
\[
 H_{\partial M}^{\widetilde g}
 =e^{-\phi/(k-1)}
  \bigl(H_{\partial M}+\langle\nabla\phi,\eta\rangle\bigr)
 \geq(k-1)|\mathrm dF|_{\widetilde g}.
\]
\hyperref[thm:wwz-listing]{Theorem~B} yields a constant $c>0$ and an
isometry from $(M,c\widetilde g)$ onto the Euclidean unit ball whose boundary
restriction is $F$.  Since $R_{\widetilde g}=0$, the scalar-curvature formula
gives $\nabla\phi=0$.  Thus $(M,g)$ is a Euclidean ball of radius $r$, and
for some constant $L>0$,
\[
 F^*g_{\bS^{k-1}}=L^2g_{\partial M}.
\]
The boundary inequality gives
$r\leq1$, while
\[
 (Lr)^{k-1}=|\deg F|\geq1,
\]
and $L\leq1$ give $r\geq1$.  Therefore $r=L=1$.
Thus $F$ is a boundary isometry and extends to an isometry from $(M,g)$ onto
the Euclidean unit ball.
\end{proof}

\begin{figure}[H]
\centering
\resizebox{\textwidth}{!}{%
\begin{tikzpicture}[
  slice/.style={draw=blue!45!black,fill=blue!8,rounded corners=2pt,
    minimum height=1.35cm,text width=3.2cm,align=center,inner sep=4pt},
  bottom/.style={draw=green!45!black,fill=green!8,rounded corners=2pt,
    minimum height=1.65cm,text width=4.2cm,align=center,inner sep=4pt}
]
 \node[slice] (M) at (0,0) {$M$\\ dimension $k+m$};
 \node[slice] (S) at (4.6,0) {$\Sigma$\\ dimension $k+m-1$};
 \node (dots) at (8.0,0) {$\cdots$};
 \node[bottom] (B) at (12.3,0)
   {final slice\\ dimension $k$\\ Euclidean $k$-ball};

 \draw[-{Latex[length=2mm]}] (M) --
   node[above=8mm]{slice for $\theta_m$} (S);
 \draw[-{Latex[length=2mm]}] (S) -- (dots);
 \draw[-{Latex[length=2mm]}] (dots) --
   node[above=8mm]{slice for $\theta_1$} (B);
 \draw[dashed,brown!75!black,-{Latex[length=2mm]}]
   (B.south) .. controls +(0,-1.4) and +(0,-1.4) ..
   node[below]{reverse induction through the free-boundary foliations}
   (M.south);
\end{tikzpicture}%
}
\caption{The dimension descent in the proof of
\cref{thm:ball-torus}.  Each free-boundary weighted slice removes one circle
direction and preserves the nonzero boundary integral.  At the end of the
$m$ steps,
\cref{lem:weighted-ball} identifies the final slice with the Euclidean
$k$-ball.  The dashed arrow represents the reverse induction that recovers
equality on the preceding slices.}
\label{fig:ball-torus-descent}
\end{figure}
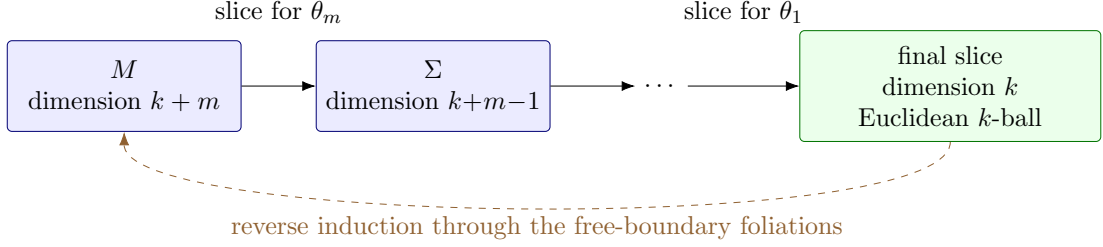

\begin{proof}[Proof of \cref{thm:ball-torus}]
We use induction on $m$.  The case $m=0$ is
\cref{lem:weighted-ball}.  Suppose $m\geq1$, and write
\[
 \Psi=(\theta_1,\ldots,\theta_m),\qquad
 \Theta_j=\theta_j^*\mathrm{dvol}_{\bS^1},\qquad
 \Omega=F^*\mathrm{dvol}_{\bS^{k-1}}.
\]
The degree hypothesis gives
\[
 \int_{\partial M}
 \Omega\wedge\Theta_1\wedge\cdots\wedge\Theta_m\ne0.
\]
Apply \cref{lem:one-step-weighted-slicing} with
\[
 \alpha=\Omega\wedge\Theta_1\wedge\cdots\wedge\Theta_{m-1},
 \qquad \theta=\theta_m.
\]
We obtain a connected weighted area minimizer $\Sigma\subset M$, a positive
function $u$, and $\lambda\geq0$, with
\begin{equation}\label{eq:first-slice-degree}
 \int_{\partial\Sigma}
 \Omega\wedge\Theta_1\wedge\cdots\wedge\Theta_{m-1}\ne0.
\end{equation}
Define $\psi=\varphi|_\Sigma+\log u$.  The identities in
\cref{lem:one-step-weighted-slicing} imply
\[
 -2\Delta_\Sigma\psi-|\nabla_\Sigma\psi|^2+R_\Sigma\geq0,
 \qquad
 H_{\partial\Sigma}
 +\langle\nabla_\Sigma\psi,\eta\rangle\geq k-1,
 \qquad
 |\mathrm d(F|_{\partial\Sigma})|\leq1.
\]
The induction hypothesis applied to $\Sigma$ and the first $m-1$ circle maps
shows that $\psi$ is constant and that $\Sigma$ is locally isometric to
$\bB^k\times\mathbb T^{m-1}$.  Thus $R_\Sigma=0$, and the exact scalar
identity gives
\[
 -2\Delta_M\varphi-|\nabla\varphi|^2+R_g=0,\quad
 \nabla_\Sigma\log u=0,\quad A_\Sigma=0,\quad\lambda=0
\]
along $\Sigma$.  Hence $u$ is constant and
$\nabla_\Sigma\varphi=0$.  Weighted minimality and $A_\Sigma=0$ give
$\langle\nabla\varphi,\nu\rangle=0$.  Therefore
$\nabla\varphi=0$ on $\Sigma$.  The eigenvalue equation and boundary
condition now give
\[
 \Ric_g(\nu,\nu)=\Hess_g\varphi(\nu,\nu),
 \qquad h_{\partial M}(\nu,\nu)=0.
\]
Thus the weighted Jacobi operator, with its free-boundary condition, is the
Neumann Laplacian.

The weighted free-boundary implicit-function theorem
\cite[Lemma~3.5]{YWang} gives a foliation
$\{\Sigma_t\}_{|t|<\delta}$ by free-boundary hypersurfaces.  Choose smooth
embeddings $\iota_t:\Sigma\to M$ with image $\Sigma_t$, and set
\[
 f_t=\langle\partial_t\iota_t,\nu_t\rangle>0.
\]
The weighted mean curvature
\[
 \mu(t):=H_{\Sigma_t}+\langle\nabla\varphi,\nu_t\rangle
\]
is constant on each slice.  The slices are relatively homologous to
$\Sigma$, so the integral in \eqref{eq:first-slice-degree} is preserved.
The Jacobi equation and its free-boundary condition
\cite[equations~(2)--(3) in the proof of Lemma~3.6]{YWang} give
\[
\begin{split}
 -\Delta_{\Sigma_t}f_t
 &-\Ric_g(\nu_t,\nu_t)f_t-|A_{\Sigma_t}|^2f_t
 +\Hess_g\varphi(\nu_t,\nu_t)f_t\\
 &\qquad-\langle\nabla_{\Sigma_t}\varphi,
                    \nabla_{\Sigma_t}f_t\rangle
 =\mu'(t),
 \qquad
 \partial_\eta f_t=h_{\partial M}(\nu_t,\nu_t)f_t.
\end{split}
\]
For $\psi_t=\varphi|_{\Sigma_t}+\log f_t$, the weighted hypersurface identity
\cite[Proposition~3.2]{YWang}, combined with the preceding Jacobi equation,
gives the Gauss and boundary identities
\begin{align}
 -2\Delta_{\Sigma_t}\psi_t-|\nabla_{\Sigma_t}\psi_t|^2+R_{\Sigma_t}
 ={}&
 \bigl(-2\Delta_M\varphi-|\nabla\varphi|^2+R_g\bigr)|_{\Sigma_t}
 \notag\\
 &+|\nabla_{\Sigma_t}\log f_t|^2+|A_{\Sigma_t}|^2
 +\mu(t)^2+\frac{2\mu'(t)}{f_t},
 \label{eq:foliation-scalar}\\
 H_{\partial\Sigma_t}
 +\langle\nabla_{\Sigma_t}\psi_t,\eta\rangle
 ={}&H_{\partial M}+\langle\nabla\varphi,\eta\rangle.
 \notag
\end{align}

If $\mu'(t_0)>0$, then $\mu'(t_0)/f_{t_0}\geq\tau$ for some $\tau>0$.
The induction hypothesis applied to $(\Sigma_{t_0},\psi_{t_0})$ implies that
$\psi_{t_0}$ is constant and $\Sigma_{t_0}$ is locally flat.  The left-hand
side of \eqref{eq:foliation-scalar} is therefore zero, whereas its right-hand
side is at least $2\tau$.  Hence
$\mu'\leq0$.  Since $\mu(0)=0$, the first variation gives
\[
 \frac{\mathrm d}{\mathrm dt}
 \int_{\Sigma_t}e^\varphi\,\mathrm dA_g
 =\mu(t)\int_{\Sigma_t}e^\varphi f_t\,\mathrm dA_g.
\]
For $t>0$, we have $\mu(t)\leq0$, whereas for $t<0$, we have
$\mu(t)\geq0$.  Integrating the first-variation formula from $0$ to $t$
shows in both cases that the weighted area of $\Sigma_t$ is at most that of
$\Sigma$.  The reverse inequality follows because $\Sigma$ minimizes
weighted area in its relative homology class.  Thus equality holds,
$\mu=\mu'=0$, and every $\Sigma_t$ is also a weighted area minimizer.  The
induction hypothesis implies that $\psi_t$ is constant, and
\eqref{eq:foliation-scalar} gives
$\nabla_{\Sigma_t}f_t=0$ and $A_{\Sigma_t}=0$.  Hence
$\nabla_{\Sigma_t}\varphi=0$; moreover, $\mu(t)=0$ and
$H_{\Sigma_t}=0$ give $\langle\nabla\varphi,\nu_t\rangle=0$.  Therefore
\[
 A_{\Sigma_t}=0,\qquad
 \nabla_{\Sigma_t}f_t=0,\qquad
 \nabla\varphi=0
\]
on every slice.  Reparametrizing the family, we obtain a Riemannian product
neighborhood of $\Sigma$.

We next extend this product structure to all of $M$, following the
continuation argument of \cite[Theorem~3.10]{YWang}.  Reparametrize by signed
distance, write $\Phi(x,s)=\exp_x(s\nu(x))$, and let $I$ be the maximal open
interval containing $0$ such that
\[
 \Phi:\Sigma\times I\longrightarrow M
\]
is a boundary-preserving local isometry and $\varphi\circ\Phi$ is constant.
Set
\[
 \alpha=\Omega\wedge\Theta_1\wedge\cdots\wedge\Theta_{m-1},
 \qquad \widehat g=e^{2\varphi/(N-1)}g.
\]
Suppose that $b=\sup I<\infty$.  Then $\Phi_s$ converges smoothly to an
isometric immersion $\Phi_b:\Sigma\to M$.  Let
$T_s=(\Phi_s)_\#[\![\Sigma]\!]$.  These currents represent the same relative
homology class, and the minimizing property of $\Sigma$ and the area formula
give
\[
 \mathbf M_{\widehat g}(T_0)
 \leq\mathbf M_{\widehat g}(T_b)
 \leq\int_\Sigma J_{\widehat g}\Phi_b\,\mathrm dA_{\widehat g}
 =\mathbf M_{\widehat g}(T_0).
\]
The free-boundary regularity theorem
\cite[Appendix~B, Theorem~B.1]{BrendleHung}, the equality in the area
formula, and the boundary maximum principle show that
\[
 T_b=q[\![\Gamma]\!],\qquad
 \Phi_b:\Sigma\longrightarrow\Gamma
\]
for a connected smooth embedded free-boundary hypersurface $\Gamma$ and an
integer $q\geq1$, and that $\Phi_b$ is an orientation-preserving
$q$-sheeted covering.  In particular,
\[
 0\ne\partial T_b(\alpha)=q\int_{\partial\Gamma}\alpha.
\]
Moreover, $\Gamma$ minimizes $\widehat g$-area in its relative homology
class; otherwise $q$ copies of a smaller competitor would contradict the
minimality of $T_b$.  Thus the induction hypothesis applies to $\Gamma$.
Passing to the limit as $s\uparrow b$ in the geometric identities for
$\Phi_s(\Sigma)$ gives
\[
 A_\Gamma=0,\qquad \nabla\varphi=0,\qquad
 \Ric_g(\nu,\nu)=\Hess_g\varphi(\nu,\nu),\qquad
 h_{\partial M}(\nu,\nu)=0.
\]
The preceding foliation argument, now applied to $\Gamma$, gives a
boundary-preserving product local isometry
\[
 E:\Gamma\times(-\varepsilon,\varepsilon)\longrightarrow M.
\]
Because the covering preserves orientation, its sheets induce the same unit
normal on $\Gamma$.  Geodesic uniqueness therefore gives
\[
 \Phi(x,s)=E(\Phi_b(x),s-b),
\]
near $b$, contradicting the maximality of $I$.  The same argument at the
lower endpoint proves that $I=\mathbb R$.  Since $\Sigma\times\mathbb R$ is
complete, the boundary-preserving local isometry
$\Phi:\Sigma\times\mathbb R\to M$ is a covering map.  It is surjective
because $M$ is connected, and hence $\varphi$ is constant on $M$.

The local isometry from $\Sigma$ to
$\overline{\bB}^k\times\mathbb T^{m-1}$ given by the induction hypothesis
has open image.  Since $\Sigma$ is compact, its image is also compact and
hence closed in the connected target
$\overline{\bB}^k\times\mathbb T^{m-1}$.  It is therefore surjective; being a
proper local diffeomorphism, it is a finite covering.  Hence the universal
cover of $\Sigma$ is isometric to
$\overline{\bB}^k\times\mathbb R^{m-1}$.  Since
$\Phi:\Sigma\times\mathbb R\to M$ is a covering, the universal cover of
$M$ is therefore
\[
 (\widetilde M,\widetilde g)
 \cong(\overline{\bB}^k\times\mathbb R^m,g_{\mathrm{Euc}}).
\]
Let $\pi:\widetilde M\to M$ be the covering map.  Equality on every leaf
and $|\mathrm dF|\leq1$ show that, on $\partial\widetilde M$, $F\circ\pi$
agrees with the spherical projection up to an isometry of
$\bS^{k-1}$.  Hence every covering transformation acts trivially on the
ball factor, and
\[
 (M,g)\cong\overline{\bB}^k\times Q
\]
for a compact flat $m$-manifold $Q$.  Fiber integration of the degree of
$(F,\Psi|_{\partial M})$ gives a map $Q\to\mathbb T^m$ of nonzero degree,
and therefore $b_1(Q)\geq m$.  By Bieberbach theory, this implies that $Q$
has trivial holonomy and hence is a flat torus.  After identifying $Q$ with
$(\mathbb T^m,g_0)$, the product isometry gives $\mathcal I$, whose first
component is the required extension of $F$.
\end{proof}

\appendix

\section{The two-dimensional cases}
\label{sec:two-dimensional}

\begin{proof}[Proof of \cref{thm:main} for $n=2$]
Write $K_g=R_g/2$, and denote the area of $M$ and the length of $\partial M$
by $A$ and $L$, respectively.  Let $\nu$ be the outward unit normal, and let
$\kappa=H_{\partial M}$ denote the geodesic curvature with the same sign
convention.  For $p\in\partial M$, define the inward normal geodesic by
$\gamma_p(t)=\exp_p(-t\nu)$, and let $\gamma_p(\tau(p))$ be its cut
point.  Denote the normal Jacobian by $J(p,t)$.  For $0\leq t<\tau(p)$,
\[
 J_{tt}+K_g(\gamma_p(t))J=0,
 \qquad J(p,0)=1,
 \qquad J_t(p,0)=-\kappa(p).
\]
Let $j_\rho(t)=\sinh(\rho-t)/\sinh\rho$.  For $w=J-j_\rho$,
\[
 w_{tt}-w=-(K_g+1)J\leq0,
 \qquad w(p,0)=0,
 \qquad w_t(p,0)\leq0.
\]
Since
\[
 \bigl(e^t(w_t-w)\bigr)_t=e^t(w_{tt}-w)\leq0,
\]
the initial conditions give $w_t-w\leq0$.  Hence
$(e^{-t}w)_t\leq0$, and therefore $J(p,t)\leq j_\rho(t)$.  The positivity of
$J$ for $0\leq t<\tau(p)$ and
$j_\rho(\rho)=0$ imply $\tau(p)\leq\rho$.  The boundary cut locus has area
zero, so
\[
 A=\int_{\partial M}\int_0^{\tau(p)}J(p,t)\,\md t\,\md s
 \leq L\int_0^\rho j_\rho(t)\,\md t
 =L\tanh\frac{\rho}{2}.
\]
If $d=\deg f$, then
\[
 L\geq\int_{\partial M}|\md f|\,\md s
 \geq2\pi|d|\sinh\rho.
\]
Since $\partial M$ is connected,
$\chi(M)=1-2\,\operatorname{genus}(M)\leq1$.  Gauss--Bonnet gives
\[
 \begin{split}
 2\pi\chi(M)
 &=\int_MK_g\,\md A+\int_{\partial M}\kappa\,\md s\\
 &\geq-A+L\coth\rho
 \geq\frac{L}{\sinh\rho}
 \geq2\pi|d|
 \geq2\pi.
 \end{split}
\]
Thus equality holds throughout.  In particular,
\[
 \operatorname{genus}(M)=0,
 \qquad |d|=1,
 \qquad L=2\pi\sinh\rho,
 \qquad K_g\equiv-1,
 \qquad \kappa\equiv\coth\rho.
\]
Equality also gives
\[
 A=L\tanh\frac{\rho}{2}=2\pi(\cosh\rho-1),
 \qquad L^2=A^2+4\pi A.
\]
Since $M$ is a topological disk and $K_g\equiv-1$, the equality case in
\cite{ChavelFeldman} shows that $(M,g)$ is isometric to
$\overline{\bB}^{\bH^2}_\rho$.  Equality in
$L\geq\int_{\partial M}|\md f|\,\md s$, together with
$\operatorname{Lip}(f)\leq1$, gives $|\md f|\equiv1$.  Since $|d|=1$, the
map $f$ is a boundary isometry and extends to an isometry
$(M,g)\to\overline{\bB}^{\bH^2}_\rho$.
\end{proof}

\begin{proof}[Proof of \cref{thm:listing} for $n=2$]
Again write $K_g=R_g/2$, $A=\operatorname{Area}(M)$, and
$L=\operatorname{Length}(\partial M)$.  Let $\kappa$ denote the geodesic
curvature of $\partial M$, and define
\[
 q=\sqrt{\kappa^2-1}.
\]
For $p\in\partial M$, use the notation $\tau(p)$ and $J(p,t)$ from the proof
of \cref{thm:main}.  Let
\[
 r(p)=\operatorname{arccoth}\kappa(p),
 \qquad
 j_p(t)=\cosh t-\kappa(p)\sinh t.
\]
Here $r(p)=\infty$ when $\kappa(p)=1$.
The comparison argument in the proof of \cref{thm:main} gives
$J(p,t)\leq j_p(t)$ for
$0\leq t<\tau(p)$.  Since $j_p(r(p))=0$, it follows that
$\tau(p)\leq r(p)$.  Therefore
\[
 \begin{split}
 A
 &\leq\int_{\partial M}\int_0^{r(p)}j_p(t)\,\md t\,\md s\\
 &=\int_{\partial M}(\kappa-q)\,\md s.
 \end{split}
\]
Gauss--Bonnet and \eqref{eq:listing-assumption} imply
\[
 \begin{split}
 2\pi\chi(M)
 &\geq-A+\int_{\partial M}\kappa\,\md s
 \geq\int_{\partial M}q\,\md s\\
 &\geq\int_{\partial M}|\md F|\,\md s
 \geq2\pi|\deg F|
 \geq2\pi.
 \end{split}
\]
Since $\partial M$ is connected,
$\chi(M)=1-2\,\operatorname{genus}(M)\leq1$.  Hence equality holds throughout.
Thus $M$ is a topological disk,
\[
 K_g\equiv-1,
 \qquad |\deg F|=1,
 \qquad q=|\md F|,
 \qquad \int_{\partial M}q\,\md s=2\pi.
\]
Moreover, Gauss--Bonnet now gives
$\int_{\partial M}\kappa\,\md s=2\pi+A$.  The function
$x\mapsto\sqrt{x^2-1}$ is strictly concave on $[1,\infty)$.  Jensen's
inequality gives
\[
 \begin{split}
 2\pi
 =\int_{\partial M}q\,\md s
 &\leq L\sqrt{\left(\frac{2\pi+A}{L}\right)^2-1}\\
 &=\sqrt{(2\pi+A)^2-L^2}.
 \end{split}
\]
The Bol--Fiala inequality \cite{ChavelFeldman} gives
\[
 L^2\geq A^2+4\pi A,
\]
with equality only for a hyperbolic geodesic disk.  Consequently
$\sqrt{(2\pi+A)^2-L^2}\leq2\pi$, and equality holds in both inequalities.
It follows that $(M,g)$ is a hyperbolic geodesic disk and $\kappa$ is
constant.  Write its radius as $\operatorname{arsinh}a$.  Then
\[
 \kappa=\sqrt{1+a^{-2}},
 \qquad q=a^{-1}.
\]
Since $q=|\md F|$ and $|\deg F|=1$, the map
$aF:\partial M\to\bS^1(a)$ is a boundary isometry and extends to an
isometry from $(M,g)$ onto
$\overline{\bB}^{\bH^2}_{\operatorname{arsinh}a}$.  Equality holds in
\eqref{eq:listing-assumption}.
\end{proof}

\section{A warped-torus counterexample to hyperbolic stabilization}
\label{app:warped-torus}

\begin{prop}\label{prop:hyperbolic-ball-torus-counterexample}
Let $k\geq2$, $m\geq1$, and $\rho>0$.  Fix a flat metric $\tau$ on
$\mathbb T^m$, and let $h$ be the hyperbolic metric on
$\overline{\bB}^{\bH^k}_\rho$.  On
\[
 X=\overline{\bB}^{\bH^k}_\rho\times\mathbb T^m
\]
there is a smooth metric $G$ such that
\begin{enumerate}[label=\textup{(\roman*)}]
\item $R_G\geq-k(k-1)$;
\item $G=h+\tau$ in a neighborhood of $\partial X$, and hence
      $H_{\partial X}=(k-1)\coth\rho$;
\item for the product projections
\[
 F:\partial X\to\bS^{k-1}(\sinh\rho),
 \qquad \Psi:X\to\mathbb T^m,
\]
the map $(F,\Psi|_{\partial X})$ has degree one with respect to the product
orientations, and $|\mathrm dF|=1$;
\item $G$ is not locally isometric to $h+\tau$.
\end{enumerate}
\end{prop}

\begin{proof}
Choose $0\not\equiv f\in C_c^\infty(\bB^{\bH^k}_\rho)$ and define $q=f^2$.
Then $q\geq0$, $q$ vanishes in a boundary collar, and
\[
 |\nabla_hq|^2\leq C_0q,
\]
for some $C_0>0$.  For $\varepsilon>0$, define
\[
 G_\varepsilon
 =e^{2\varepsilon q/(k-1)}h+e^{-2\varepsilon q/m}\tau.
\]
View $G_\varepsilon$ as a warped product with base metric
$e^{2\varepsilon q/(k-1)}h$ and fiber metric
$e^{-2\varepsilon q/m}\tau$.  The warped-product and conformal
scalar-curvature formulas give
\[
 R_{G_\varepsilon}
 =e^{-2\varepsilon q/(k-1)}
 \left[-k(k-1)-\left(\frac1{k-1}+\frac1m\right)
 \varepsilon^2|\nabla_hq|^2\right].
\]
Since $e^s-1\geq s$ for $s\geq0$,
\begin{align*}
 R_{G_\varepsilon}+k(k-1)
 &=e^{-2\varepsilon q/(k-1)}
 \left[k(k-1)\left(e^{2\varepsilon q/(k-1)}-1\right)
       -\left(\frac1{k-1}+\frac1m\right)
        \varepsilon^2|\nabla_hq|^2\right]\\
 &\geq\varepsilon q\,e^{-2\varepsilon q/(k-1)}
       \left[2k-\left(\frac1{k-1}+\frac1m\right)C_0\varepsilon\right].
\end{align*}
Thus $R_{G_\varepsilon}\geq-k(k-1)$ if
$0<\varepsilon\leq
2k\big/\bigl((\frac1{k-1}+\frac1m)C_0\bigr)$.  Since $q$ vanishes near $\partial X$,
properties \textup{(ii)} and \textup{(iii)} follow.  At a point $p$ where
$q$ attains its positive maximum,
\[
 R_{G_\varepsilon}(p)
 =-k(k-1)e^{-2\varepsilon q(p)/(k-1)}>-k(k-1).
\]
Since $R_{h+\tau}\equiv-k(k-1)$, property \textup{(iv)} follows.  Take
$G=G_\varepsilon$.
\end{proof}

\end{document}